\documentclass{amsart}
\usepackage{latexsym,amsxtra,amscd,ifthen,amsmath,color}
\usepackage{amsfonts}
\usepackage{verbatim}
\usepackage{amsmath}
\usepackage{amsthm}
\usepackage{amssymb}
\usepackage[notcite,notref,final]{showkeys}
\usepackage{url}

\theoremstyle{plain}
\newtheorem{maintheorem}{Theorem}

\newtheorem{theorem}{Theorem}
\newtheorem{lemma}[theorem]{Lemma}
\newtheorem{proposition}[theorem]{Proposition}

\numberwithin{theorem}{section}
\numberwithin{equation}{theorem}

\theoremstyle{definition}
\newtheorem{definition}[theorem]{Definition}
\newtheorem{example}[theorem]{Example}
\newtheorem{remark}[theorem]{Remark}
\newtheorem{question}[theorem]{Question}
\newtheorem*{question*}{Question}

\newcommand{\cwlt}{(\textup{cwlt})}
\newcommand{\INT}{\textup{int}}

\newcommand{\Z}{\mathbb{Z}}
\DeclareMathOperator{\ch}{char}
\DeclareMathOperator{\rk}{rk}
\DeclareMathOperator{\Af}{Af}
\DeclareMathOperator{\End}{End}
\DeclareMathOperator{\Aut}{Aut}
\DeclareMathOperator{\gr}{gr}
\DeclareMathOperator{\af}{af}
\DeclareMathOperator{\tr}{tr}
\DeclareMathOperator{\triangular}{tr} 
\DeclareMathOperator{\op}{op}
\DeclareMathOperator{\reg}{reg}
\DeclareMathOperator{\pr}{pr}
\DeclareMathOperator{\GKdim}{GKdim}
\DeclareMathOperator{\Kdim}{Kdim}
\DeclareMathOperator{\Disc}{Disc}

\begin{document}

\title[Discriminant controls automorphism groups]
{The discriminant controls automorphism\\
groups of noncommutative algebras}

\author{S. Ceken, J. H. Palmieri, Y.-H. Wang and J. J. Zhang}

\address{Ceken: Department of Mathematics, Akdeniz University, 07058 Antalya,
Turkey}

\email{secilceken@akdeniz.edu.tr}

\address{Palmieri: Department of Mathematics, Box 354350,
University of Washington, Seattle, Washington 98195, USA}

\email{palmieri@math.washington.edu}

\address{Wang: Department of Applied Mathematics,
Shanghai University of Finance and
Economics, Shanghai 200433, China}

\email{yhw@mail.shufe.edu.cn}

\address{Zhang: Department of Mathematics, Box 354350,
University of Washington, Seattle, Washington 98195, USA}

\email{zhang@math.washington.edu}

\begin{abstract}
We use the discriminant to determine the automorphism groups of some
noncommutative algebras, and we prove that a family of noncommutative
algebras has tractable automorphism groups.
\end{abstract}





\maketitle


\setcounter{section}{-1}
\section{Introduction}
\label{xxsec0}

There is a long history and an extensive study of the
automorphism groups of algebras. Determining the full automorphism
group of an algebra is generally a notoriously difficult problem. For
example, the automorphism group of the polynomial ring of three
variables is not yet understood, and a remarkable  result in this direction
is given by Shestakov-Umirbaev \cite{SU} which shows the
Nagata automorphism is a wild automorphism. Since 1990s,
many researchers have been
successfully computing the automorphism groups of
interesting  infinite-dimensional noncommutative algebras,
including certain quantum
groups, generalized quantum Weyl algebras, skew polynomial rings and
many more -- see \cite{AlC, AlD, AnD, BJ, GTK, SAV},
which is only a partial list. Recently, by using a rigidity theorem
for quantum tori, Yakimov has proved the Andruskiewitsch-Dumas
conjecture and the Launois-Lenagan conjecture in \cite{Y1, Y2}, each of
which determines the automorphism group of a family of quantized
algebras with parameter $q$ being not a root of unity. A uniform approach
to both the Andruskiewitsch-Dumas conjecture and the Launois-Lenagan
conjecture  is provided in a  preprint by Goodearl-Yakimov \cite{GY}.
These beautiful results, as well as others, motivated us to look
into the automorphism groups of noncommutative algebras.

To warm up, let us consider an explicit example.
For the rest of the introduction, let $k$ be a field and let
$k^\times=k\setminus\{0\}$. For any integer $n\geq 2$,
let $W_n$ be the $k$-algebra generated by
$\{x_1,\dots,x_n\}$, subject to the relations $x_ix_j+x_jx_i=1$ for
all $i\neq j$. The action of the symmetric group $S_n$ on the set
$\{x_1,\dots,x_n\}$ extends to an action of $S_n$ on the algebra $W_n$,
and the map $x_i\mapsto -x_i$  determines an algebra automorphism
of $W_n$. Therefore $S_n\times \{\pm 1\}$ is a subgroup of
the full  automorphism group $\Aut(W_n)$ of the
$k$-algebra $W_n$.  We compute $\Aut(W_n)$ when $n$ is even.

\begin{maintheorem}\label{xxthm0.1}
Assume that $\ch k \neq 2$. If $n\geq 4$ is even, then
$\Aut(W_n)=S_n \times \{\pm 1\}$.
\end{maintheorem}

It is well-known that $\Aut(W_2)=S_2\ltimes k^\times$, see
\cite{AlD}.
If $n$ is odd or $\ch k=2$, then $\Aut(W_n)$ is
unknown and contains more automorphisms than $S_n\times \{\pm 1\}$:
see Example \ref{xxex5.12}.

Understanding the automorphism group of an algebra is fundamentally
important in general, and for the algebra $W_n$, is the first
step in the study of the invariant theory under group
actions \cite{CPWZ1}. The invariant theory of $W_2$ was studied in
\cite{CWWZ}, and \cite[Theorem 0.4]{CWWZ} applies to $W_2$ as $W_2$
is filtered Artin-Schelter regular of dimension 2. We have the following
for even integers $n\geq 4$.

\begin{maintheorem}\cite{CPWZ1} \label{xxthm0.2}
Assume that $\ch k \neq 2$.
Let $n$ be an even integer $\geq 4$ and $G$ be any group acting
on $W_n$. Then the fixed subring $W_n^G$ under the $G$-action is
filtered Artin-Schelter Gorenstein.
\end{maintheorem}

By Theorem \ref{xxthm0.2}, the $W_n$'s form a class of rings
with good homological properties under any group action. The proof
of Theorem \ref{xxthm0.2} is heavily dependent on the structure of
$\Aut(W_n)$.

As stated in the first sentence of \cite{Y1}, the automorphism group
of an algebra is often difficult to describe. For an algebra with many
generators, it is usually impossible to compute its automorphism
group directly. This leads us to consider the following question.

\begin{question*}
What invariants of an algebra control its automorphism group?
\end{question*}

This question has been implicitly asked by many authors, for example,
in the papers mentioned in the first paragraph of the introduction, and
different techniques have been used in the study of automorphism
groups. In this paper, we use the \emph{discriminant}.  When $n$ is
even, the discriminant of $W_n$ over its center is a non-unit element
of the center, and it is preserved by any algebra automorphism of
$W_n$. This is how we prove Theorem \ref{xxthm0.1}. Unfortunately,
when $n$ is odd or when the characteristic of $k$ is 2, the
discriminant of $W_n$ over its center is (conjecturally) trivial,
whence no useful
information can be derived from this invariant. This is one
reason why the form of $\Aut(W_n)$ is dependent on the parity of $n$
and $\ch k$.

Our main theorem is an abstract version of Theorem \ref{xxthm0.1}.
Let $A$ be a filtered algebra with filtration $\{F_iA\}_{i\geq 0}$
such that the associated graded ring $\gr A$ is connected graded.
We say an automorphism $g\in \Aut(A)$ is
\emph{affine} if $g(F_1 A)\subset F_1 A$. Let $\Af$ be the
category of $k$-algebras $A$ satisfying the following conditions:
\begin{enumerate}
\item
$A$ is a filtered algebra such that the associated graded ring
$\gr A$ is a domain,
\item
$A$ is a finitely generated free module over its center $R$, and
\item
the discriminant $d(A/R)$ is dominating (see
Definition \ref{xxdef2.3}).
\end{enumerate}
The morphisms in this category are just isomorphisms of algebras.
Conditions (1) and (2) are easy to understand, while the terminology
in condition (3) will be defined in Sections \ref{xxsec1} and \ref{xxsec2}.
At this point we only mention that the algebras
$W_n$ are in $\Af$ when $n$ is even and that there are algebras such  that (1)
and (2) hold and (3) fails [Example \ref{xxex5.9}].

\begin{maintheorem}
\label{xxthm0.3} Let $A$ be in the category $\Af$.
In parts \textup{(3,4)}, assume that $\ch k=0$.
Let $R$ be the center of $A$. Then the following hold.
\begin{enumerate}
\item
Every automorphism $g$ of $A$ is affine.
\item
Every automorphism $h$ of the polynomial extension $A[t]$
is triangular. That is, there is a $g\in \Aut(A)$, $c\in k^\times$ and
$r\in R$ such that
\[
h(t)=ct+r \quad {\text{and}}\quad
h(x)=g(x)\in A \quad {\text{for all $x\in A$}}.
\]
In other words,
\[
\Aut(A[t])=\begin{pmatrix} \Aut(A)& R\\ 0& k^\times\end{pmatrix}.
\]
\item
Every locally nilpotent derivation (defined after
Lemma~\ref{xxlem3.2}) of $A$ is zero.
\item
$\Aut(A)$ is an algebraic group that fits into the exact sequence
\begin{equation}
 \label{0.3.1}\tag{*}
 1\to (k^\times)^r\to \Aut(A)\to S\to 1
\end{equation}
where $r\geq 0$ and $S$ is a finite group. In other words, $\Aut(A)=S\ltimes
(k^\times)^r$.
\end{enumerate}
\end{maintheorem}

If $\ch k\neq 0$, part (3) of the above could fail, see Example \ref{xxex3.9}.
Note that parts (3,4) are consequences of part (2) [Lemmas \ref{xxlem3.3}(2)
and \ref{xxlem3.4}].
Part (3) suggests that the discriminant controls locally
nilpotent derivations too.
Part (4) gives a structure theorem for $\Aut(A)$.
The integer $r$ is called the \emph{symmetry rank} of $A$,
denoted by $sr(A)$;
and the order $|S|$ is called the \emph{symmetry index} of
$A$, denoted by
$si(A)$. For example, Theorem \ref{xxthm0.1} says that,
when $n\geq 4$ is even, $sr(W_n)=0$
and $si(W_n)=2n!$.

Theorem \ref{xxthm0.3}(1) provides a uniform approach to the automorphism
groups of all algebras in $\Af$. There are many algebras in the
category $\Af$ [Section \ref{xxsec5}]. For example, if $A$ is a
PI skew polynomial ring
$k_{p_{ij}}[x_1,\dots,x_n]$ such that (a) $x_i$ is not in the center
of $A$ for all $i$ and (b) $A$ is free over its center, then $A$ is in $\Af$
\cite{CPWZ2}. Here a PI algebra  means an algebra satisfying a
\emph{polynomial identity} \cite[Chapter 13]{MR}.
The category $\Af$ also has the nice property that it is closed
under the tensor product [Theorem \ref{xxthm5.5}].

As we will see below, the discriminant method has limitations. An
immediate one is that we need to assume the existence of a ``good'' trace
function, and this does not exist for a general noncommutative
algebra -- see Example \ref{xxex1.10}.

In the sequel \cite{CPWZ2} we develop other techniques
for computing discriminants and automorphism groups. One major goal of
that paper is to work with algebras which are not free over their
centers. We also deal with algebras $B$ of the following form. First,
let $A_q$ be the $q$-quantum Weyl algebra
generated by $x$ and $y$ subject to the relation $yx=qxy+1$ for some
$q\in k^\times$ (we assume that $q\neq 1$, but $q$ need not be a root
of unity). Consider the tensor product $B:=A_{q_1}\otimes \cdots
\otimes A_{q_m}$ of quantum Weyl algebras, where $q_i\in k^\times
\setminus \{1\}$ for all $i$.  Since we are
not assuming that the $q_i$ are roots of unity, $B$ need not be in $\Af$;
however, the conclusion of Theorem \ref{xxthm0.3} holds for $B$:

\begin{maintheorem}\cite{CPWZ2}
\label{xxthm0.4} Let $B=A_{q_1}\otimes \cdots \otimes A_{q_m}$ and
assume that $q_i\neq 1$ for all $i=1,\dots,m$.
\begin{enumerate}
\item
The automorphism group $\Aut(B)$ is an algebraic group that
fits into an exact sequence of the form \eqref{0.3.1}.
\item
The automorphism group of $B[t]$ is triangular, namely,
\[
\Aut(B[t])=\begin{pmatrix} \Aut(B) & C(B)\\0& k^\times\end{pmatrix}
\]
where $C(B)$ is the center of $B$.
\item
If $\ch k=0$,
then every locally nilpotent derivation of $B$ is zero.
\end{enumerate}
\end{maintheorem}

Two explicit examples are given in \cite{CPWZ2}.
Let $B$ be as in Theorem \ref{xxthm0.4}.
\begin{enumerate}
\item
If $q_i\neq \pm1$ and $q_i\neq q_j^{\pm 1}$ for all $i\neq j$, then
$\Aut(B)=(k^\times)^m$.
\item
If $q_i=q\neq \pm 1$ for all $i$, then $\Aut(B)=S_m \ltimes (k^\times)^m$.
\end{enumerate}

Theorem \ref{xxthm0.4} also holds for the tensor products of $A_{q}$'s
with $W_{n}$'s (for $n$ even), as well as with many others in $\Af$.

We would like to remark that most results in the
literature (including the papers mentioned
at the beginning of the introduction) calculate the automorphism
group of non-PI algebras, or algebras with a parameter $q$
(or multi-parameters) not being a root of unity. In general
it is more difficult to compute the
automorphism group in the PI case, or when $q$ is a root of
unity.  Our method deals with both the PI and non-PI
cases. Theorem \ref{xxthm0.3} works  for the PI case, and
then mod $p$ reduction (to be discussed in the sequel \cite{CPWZ2})
reduces the non-PI case (with appropriate parameters) to the PI case.

The definition of the discriminant is purely linear algebra, but
the computation of the discriminant seems to be very difficult and
tedious in general.  In this paper we only (partially) compute one
nontrivial example that is needed in the proof of Theorem \ref{xxthm0.1}.
It would be nice to develop basic theory and
computational tools for the discriminant in the
noncommutative setting.

The paper is laid out as follows. In Section~\ref{xxsec1}, we recall
the notion of the discriminant, and we establish some of its basic
properties. In Sections~\ref{xxsec2} and \ref{xxsec3}, we discuss
so-called ``affine''
and ``triangular'' automorphisms and prove Theorem
\ref{xxthm0.3}. The discriminant computation of $W_n$ over its
center occupies a major part of Section 4
and Theorem \ref{xxthm0.1} is proved near the end of Section 4.
In Section 5 we give comments, remarks, and examples related to
the category $\Af$.

\section{Discriminant in the noncommutative setting}
\label{xxsec1}

Throughout let $k$ be a commutative domain. Modules (sometimes called
vector spaces), algebras and morphisms are over $k$.

According to \cite{GKZ}, the discriminant for polynomials was
introduced by Cayley in 1848. Since then, it has been important in
number theory (Galois theory) and algebraic geometry.
In this section, we discuss the
concept of the discriminant in the noncommutative setting. Let $R$ be
a commutative algebra and let $B$ and $F$ be algebras both of which
contain $R$ as a subalgebra. In applications, $F$ would be either $R$
or the ring of fractions of $R$.

\begin{definition}
\label{xxdef1.1} An $R$-linear map $\tr: B\to F$ is called a
\emph{trace map} if $\tr(ab)=\tr(ba)$ for all $a,b\in B$.
\end{definition}

Here are some examples.

\begin{example}
\label{xxex1.2}
\begin{enumerate}
\item
Let $B=M_n(R)$. The \emph{internal trace} $\tr_{\INT}: B\to R$ is defined
to be the usual matrix trace, namely, $\tr_{\INT}((r_{ij}))=
\sum_{i=1}^n r_{ii}$.
\item
Let $B$ be a subalgebra of $M_n(F)$ and $R$ a subalgebra of
$F\cap B\subset M_n(F)$. The composition $\tr: B\to M_n(F)
\xrightarrow{\tr_{\INT}} F$ is a trace map from $B$ to $F$.
\item
Let $B$ be an $R$-algebra and $F$ be a commutative $R$-subalgebra of
$B$ such that $B_F:=B\otimes_R F$
is finitely generated free over $F$. Then left
multiplication defines a natural embedding of $R$-algebras $lm: B\to
\End_F(B_F)\cong M_n(F)$ where $n$ is the rank $\rk(B/F)$. By
part (2), we obtain a trace map, called the \emph{regular trace},
by composing: $\tr_{\textup{reg}}: B\xrightarrow{lm}
M_n(F)\xrightarrow{\tr_{\INT}} F$.
\end{enumerate}
\end{example}

Although we are going to mainly use the regular trace in this paper,
the definition of the discriminant works for any trace map. From now
on, assume that $F$ is a commutative algebra. Let $R^\times$ be the
set of units in $R$. For any $f,g\in R$, we use the notation
$f=_{R^{\times}} g$ to indicate that $f=cg$ for some $c\in R^{\times}$.
The following definition can be found in Reiner's book \cite{Re}.

\begin{definition}
\label{xxdef1.3} Let $\tr: B\to F$ be a trace map and $w$ be a fixed
integer. Let $Z:=\{z_i\}_{i=1}^w$ be a subset of $B$.
\begin{enumerate}
\item
The \emph{discriminant} of $Z$ is defined to be
\[
d_w(Z:\tr)=\det(\tr(z_iz_j))_{w\times w}\in F.
\]
\item \cite[Section 10, p.~126]{Re}.
The \emph{$w$-discriminant ideal} (or \emph{$w$-discriminant
$R$-module}) $D_w(B,\tr)$ is the $R$-submodule of $F$ generated by
the set of elements $d_w(Z:\tr)$ for all $Z=\{z_i\}_{i=1}^w\subset
B$.
\item
Suppose $B$ is an $R$-algebra which is finitely generated free over
$R$. If $Z$ is an $R$-basis of $B$, the \emph{discriminant} of $B$ is
defined to be
\[
d(B/R)=_{R^\times} d_w(Z:\tr).
\]
\item
We say the discriminant (respectively, discriminant ideal) is
\emph{trivial} if it is either 0 or 1
(respectively, it is either the zero ideal or contains 1).
\end{enumerate}
\end{definition}

The following well-known proposition establishes some basic properties of
the discriminant, including that $d(B/R)$ is independent of the
choice of $Z$.

\begin{proposition}
\label{xxpro1.4} Let $\tr: B\to R$ be an $R$-linear trace map (so $F=R$).
Let $Z:=\{z_i\}_{i=1}^w$ be a set of elements in $B$.
\begin{enumerate}
\item
\cite[p.66, Exer. 4.13]{Re} Suppose that $Y=\{y_j\}_{j=1}^w$ such that
$y_i=\sum_j r_{ij} z_j$ where $r_{ij}\in R$, and denote the matrix
$(r_{ij})_{w\times w}$ by $(Y:Z)$. Then
\[
d_w(Y:\tr)=\det(Y:Z)^2 d_w(Z:\tr).
\]
\item
If both $Y$ and $Z$ are $R$-linear bases of $B$, then
\[
d_w(Y:\tr)=_{R^{\times}}d_w(Z:\tr).
\]
As a consequence $d(B/R)$ is well-defined up to a scalar in
$R^\times$.
\item
\cite[Theorem 10.2]{Re} If $B$ is an $R$-algebra which is finitely
generated free over $R$ with an $R$-basis $Z$, then $D_w(B,\tr)$ is
the principal ideal of $R$ generated by $d_w(Z:\tr)$ or equivalently
by $d(B/R)$.
\end{enumerate}
\end{proposition}

\begin{proof} (2) is an immediate consequence of (1).
\end{proof}

Here are some simple examples. The first two indicate the connection
with the classical theory and third one is relevant to Theorem \ref{xxthm0.1}.

\begin{example}
\label{xxex1.5}
If $f$ is a monic polynomial, then its discriminant $\Disc(f)$
is classically defined to be
the product of the differences of the roots.
If $f$ is the minimal polynomial of an algebraic number
$\alpha$, it is well-known that
$d({\mathbb Z}[\alpha]/{\mathbb Z})=\Disc(f)$,
see \cite[pp. 66-67, Exer. 414 and Theorem 4.35]{Re},
or \cite[Theorem 6.4.1]{AW}, or \cite[Definition 6.2.2 and Remark 6.2.3]{St}.
\end{example}

\begin{example}
\label{xxex1.6}
Let $B = M_n(R)$. A word of caution: we are using the regular trace
map, not the internal trace map, to compute the discriminant. If we
use the basis $Z = \{e_{ij}\mid 1\leq i,j \leq n\}$ of matrix units, then we have
\[
e_{ij} e_{kl} = \begin{cases}
e_{il} & \text{if $j=k$}, \\
0 & \text{else}.
\end{cases}
\]
So we need to compute the regular trace of the matrix $e_{il}$: we compute
the trace of the matrix giving its action by left multiplication on
$M_n(R)$. Diagonal entries in that matrix arise when $e_{il} e_{jk}$
is a scalar multiple of $e_{jk}$, which can only happen when $i=l=j$;
in this case, there are $n$ diagonal entries, each of which is 1, so
\[
\tr_{\textup{reg}}(e_{ij} e_{kl}) = \begin{cases}
n & \text{if $i=l$ and $j=k$},\\
0 & \text{otherwise}.
\end{cases}
\]
Therefore $d_{n^{2}}(Z : \tr) = \pm n^{n^{2}}$.
\end{example}

\begin{example}
\label{xxex1.7}
\begin{enumerate}
\item
Let $B=W_2=k\langle x,y\rangle/(xy+yx-1)$ and let
$R=k[x^2,y^2]\subset B$. Then it is easy to check that $R$ is the
center of $B$
and $B=R\oplus Rx\oplus Ry\oplus Rxy$. Using the
regular trace $\tr$, one sees that
\[
\tr(1)=4, \quad
\tr(x)=0, \quad
\tr(y)=0, \quad
\tr(xy)=2.
\]
Using these traces and the fact $\tr$ is $R$-linear, we have the
matrix
\[
(\tr(z_iz_j))_{4\times 4}=\begin{pmatrix} 4 &0 &0 &2 \\
0& 4 x^2 &2&0\\ 0&2&4y^2&0\\2&0&0&2-4x^2y^2
\end{pmatrix}
\]
where $Z=\{z_1,z_2,z_3,z_4\}=\{1,x,y,xy\}$, and therefore the
discriminant of $d(B/R)$ is the determinant of the matrix
$(\tr(z_iz_j))_{4\times 4}$, which is, by a direct computation, $-2^4
(4x^2y^2-1)^2$.
\item
Let $C$ be the skew polynomial ring $k_{-1}[x,y]:=k\langle x,y\rangle/(xy+yx)$.
A similar computation shows that the discriminant of $C$ over its center
$R=k[x^2,y^2]$ is $-2^8 x^4y^4$. The details are left to the reader.
\end{enumerate}
\end{example}

Now we consider the case when $B$ contains a central subalgebra $R$.
Assume that $F$ is a localization of $R$ such that $B_F:=B\otimes_R F$ is
finitely generated free over $F$. For example, if $B_R$ is free, we
may take $F=R$, and if not, we may take $F$ to be the field of
fractions of $R$ (assuming $R$ is a domain).
We let $\tr_{\reg}: B\to F$ denote
the regular trace defined in Example \ref{xxex1.2}(3), namely,
\begin{equation}\label{1.7.1}\tag{1.7.1}
\tr_{\reg}: B\to B_F \xrightarrow{lm} \End_F (B_F)\xrightarrow{\tr_{\INT}} F.
\end{equation}
We also simply write $\tr$ for $\tr_{\reg}$ since this is used
most of the time.
For any algebra $B$, let $\Aut(B)$ denote the full algebra
automorphism group of $B$ over the base ring. If $C$ is a
central subalgebra of $B$, the subgroup of automorphisms
which fix $C$ is denoted $\Aut_C(B)$.
We say that an element $g \in \Aut (B)$
\emph{preserves} a subalgebra $A$ of $B$ if $g(A) \subseteq A$. Note
that if $g$ preserves $R$, then $g$ preserves any localization of
$R$, and in particular, it preserves $F$. We also note that, in case
$R$ is the center of $B$, any automorphism will preserve it.

\begin{lemma}
\label{xxlem1.8} Fix $g\in \Aut(B)$ such that $g$ and $g^{-1}$ preserve
$R$ and let $w=\rk(B_F/F)$. Let $x$ be an element in $B$.
\begin{enumerate}
\item
For any $F$-basis $Z=\{z_i\}_{i=1}^w$ of $B_F$, if $x z_i=
\sum_j r_{ij} z_j$ for some $r_{ij}\in F$,
then $\tr(x)=\sum_{i=1}^w r_{ii}$.
\item
$g(\tr(x))=\tr(g(x))$ for any $x\in B$.
\item
$g(d_w(Z:\tr))=d_w(g(Z):\tr)$ for any set $Z=\{z_i\}_{i=1}^w$.
\item
The discriminant $R$-module $D_w(B,\tr)$ is $g$-invariant.
\item
Suppose the image of $\tr$ is in $R$ and consider the trace map
$\tr: B\to R$. Then the discriminant ideal $D_w(B,\tr)$ is
$g$-invariant.
\item
If $B$ is finitely generated free over $R$, then the discriminant
$d(B/R)$ is a $g$-invariant up to a unit of $R$.
\end{enumerate}
\end{lemma}

\begin{proof} (1) This is the definition of trace, noting that
$\tr_{\INT}$ is independent of the choices of basis $Z$.

(2) If $Z=\{z_i \}_{i=1}^w$ is an $F$-basis, so is
$Y=\{g(z_i)\}_{i=1}^w$ by linear algebra. So by part (1),
we may use $Y$ to compute $\tr$. Applying $g$ to $xz_i$ we have
$g(x)g(z_i)=\sum_j g(r_{ij}) g(z_j)$. Since $g$ preserves $R$, we obtain
$\tr(g(x))=\sum_{i=1}^w  g(r_{ii})=g(\tr(x))$.

(3) This follows from part (2), the definition of $d_w(Z:\tr)$ and
an easy computation.

(4) It follows from part (3) and the definition that $g(D_w(B,\tr))
\subset D_w(B,\tr)$. Since $g$ and $g^{-1}$ are automorphisms, we have
$g(D_w(B,\tr))= D_w(B,\tr)$.

(5) This is a consequence of (4).

(6) By Proposition \ref{xxpro1.4}(3), $D_w(B,\tr)$ is a principal
ideal generated by $d(B/R)$. Since $g$ preserves $D_w(B,\tr)$,
$g(d(B/R))=cd(B/R)$ for some $c\in R^\times$.
\end{proof}

We conclude this section with a well-known observation.

\begin{example}
\label{xxex1.10} Let $k$ be a field. Let $A_1$ be the first Weyl
algebra, the algebra generated by $x$ and $y$ subject to the relation
$yx=xy+1$.

Assume first that $\ch
k=0$. Let $B$ be an algebra and let $\tr: A_1\to B$ be any additive
map such that $\tr(ab)=\tr(ba)$ for all $a,b\in A_1$.
Then $\tr(A_1)=0$, as every element
in $A_1$ can be written as $ya-ay$ for some $a\in A_1$ -- for any $m,
n \geq 0$ and any $c \in k$, we have
\[
c x^my^n = y\left(\frac{c}{m+1}x^{m+1}y^n\right) -
\left(\frac{c}{m+1}x^{m+1}y^n\right)y.
\]
So there is no nontrivial trace map from $A_1$ to any algebra.

If $\ch k=p>0$, then $A_1$ is a finitely generated free module
over its center $R:=k[x^p,y^p]$. A direct computation shows
that the regular trace $\tr: A_1\to R$ is the zero map in this case.
\end{example}

\section{Dominating elements and  automorphisms}
\label{xxsec2}

In this section, we establish tools for identifying and constructing
certain algebra automorphisms, called ``affine'' and ``triangular''
automorphisms. In
the situation of Theorem~\ref{xxthm0.1}, we can show that every
automorphism is affine -- see Section~\ref{xxsec4} -- and this allows
us to prove the theorem.

The main result in this section is Theorem~\ref{xxthm0.3}(1). To state and
prove it, we need the concept of a ``dominating element,'' which we
now develop.

Let $A$ be an algebra over $k$.
We say $A$ is \emph{connected graded} if
$A=k\oplus A_1\oplus A_2\oplus \cdots$ and $A$ is \emph{locally finite}
if each $A_i$ is finitely generated over $k$. We now consider
filtered rings. Let $Y$ be a finitely generated free $k$-submodule of
$A$. In this case we would also say that
$Y$ is finite-dimensional (over $k$). Suppose $k\cap Y=\{0\}$. Consider
the standard filtration $F=\{F_n A:= (k\oplus Y)^n\mid n\geq 0\}$ and assume
that $F$ is an exhaustive filtration of $A$ and that the associated graded
ring $\gr A$ is connected graded. As a consequence of $\gr A$ being
connected graded, the unit map $k\to A$ is injective.
For each element $f\in
F_n A\setminus F_{n-1} A$, the associated element in $\gr A$ is
defined to be $\gr f=f+F_{n-1} A\in (\gr_F A)_n$. The degree of an
element $f\in A$, denoted by $\deg f$, is defined to be the degree
of $\gr f$. By definition, $\deg c=0$ for all $0\neq c\in k$.

Using the standard filtration $\{F_n A= (k\oplus Y)^n\mid n\geq 0\}$
makes it easier to talk about
affine automorphisms [Definition \ref{xxdef2.5}]. But the ideas in this
section also apply to non-standard filtrations, see Example \ref{xxex5.8}.

Note that, if $\gr A$ is a domain, then, for any elements $f_1, f_2 \in A$,
\begin{equation}\label{2.0.1}\tag{2.0.1}
\deg (f_1 f_2) = \deg f_1 + \deg f_2.
\end{equation}

Let $A^\times$ denote the set of all units of $A$.  If $\gr A$ is a
connected graded domain, as we assume in much of what follows, it is
easy to see that $A^\times =k^\times$. In this case, if $R$ is any
subalgebra of $A$ (for example, if $R$ is the center of $A$),
$R^\times =k^\times$.

One can check that assigning degrees (which could be different from
$1$) to a set of generators of
$A$ is almost equivalent to giving a filtration on $A$, though
not every filtration has the property that $\gr A$ is a
domain. See \cite[Section 1]{YZ} for some details.

\begin{definition}
\label{xxdef2.1} Suppose that $Y=\bigoplus_{i=1}^n kx_i$ generates
$A$ as an algebra.
\begin{enumerate}
 \item
A nonzero element $f:=f(x_1,x_2,\dots,x_n)\in A$ is called
\emph{locally dominating} if, for every $g\in \Aut(A)$,
one has
\begin{enumerate}
\item
$\deg f(y_1,\dots,y_n)\geq \deg f$ where
$y_i=g(x_i)$ for all $i$, and
\item
if, further, $\deg y_{i_0}>1$ for some $i_0$, then
$\deg f(y_1,\dots,y_n)> \deg f$.
\end{enumerate}
\item
Assume that $\gr A$ is a connected graded domain.
A nonzero element $f\in A$ is called \emph{dominating} if, for every
filtered PI algebra $T$ with $\gr T$ a connected
graded domain, and for every subset of elements
$\{y_1,\dots,y_n\}\subset T$ that is linearly independent in
the quotient $k$-module $T/F_0 T$, there is a lift
of $f$, say $f(x_1,\dots,x_n)$, in the free algebra $k\langle
x_1,\dots,x_n\rangle$, such that the following hold: either
$f(y_1,\dots,y_n)=0$ or
\begin{enumerate}
\item
$\deg f(y_1,\dots,y_n)\geq \deg f$, and
\item
if, further, $\deg y_{i_0}>1$ for some $i_0$, then
$\deg f(y_1,\dots,y_n)> \deg f$.
\end{enumerate}
\end{enumerate}
\end{definition}

We refer to $T$ as a ``testing'' algebra.
To prove our main Theorem \ref{xxthm0.3}, we only need one
testing algebra, $T=A\otimes k[t]=A[t]$. But it convenient
to include all testing algebras $T$ in order to prove Theorem \ref{xxthm5.5}.
In almost all applications, it is easy to see that $f(y_1,\dots,y_n)\neq 0$;
so we only need to verify (a) and (b) in order to show that $f$ is
dominating. If this is the case, we will not mention the subcase of
$f(y_1,\dots,y_n)=0$.

It is not hard to see that dominating elements are locally dominating.
Next we give some examples of dominating elements.
A monomial $x_1^{b_1}\cdots x_n^{b_n}$ is said to have
degree \emph{component-wise less than} (or, \emph{cwlt}, for short)
$x_1^{a_1}\cdots x_n^{a_n}$ if $b_i\leq a_i$ for all $i$ and
$b_{i_0}<a_{i_0}$ for some $i_0$.  We write $f=cx_1^{b_1}\cdots
x_n^{b_n}+\cwlt$ if $f-cx_1^{b_1}\cdots x_n^{b_n}$ is a linear
combination of monomials with degree component-wise less than
$x_1^{b_1}\cdots x_n^{b_n}$.  The following is easy.

\begin{lemma}
\label{xxlem2.2} Retain the above notation and assume that $\gr A$
is a connected graded domain. Fix $f\in A$.
\begin{enumerate}
\item
If $f=cx_1^{b_1}\cdots x_n^{b_n}+\cwlt$ where $n>0$, $b_1b_2\cdots
b_n>0$, and $0\neq c\in k$, then $f$ is dominating.
\item
For any positive integer $d$, $f$ is  dominating \textup{(}respectively,
locally dominating\textup{)} if and only if $f^d$ is.
\end{enumerate}
\end{lemma}

\begin{proof} (2) is clear, using \eqref{2.0.1}. To prove
(1), write
\[
f=cx_1^{b_1}\cdots x_n^{b_n}+\sum c_{a_s}x_1^{a_1}\cdots x_n^{a_n}.
\]
Let $T$ be any ${\mathbb N}$-filtered PI domain and
$\{y_1,\cdots,y_n\}$ be a set of elements in $T$ of degree at least 1.
Suppose that $\deg y_{i_0}>1$ for some $i_0$. Since each term
$x_1^{a_1}\cdots x_n^{a_n}$ is cwlt $x_1^{b_1}\cdots x_n^{b_n}$, we
have $\deg y_1^{a_1}\cdots y_n^{a_n}< \deg y_1^{b_1}\cdots y_n^{b_n}$,
again by \eqref{2.0.1}. Hence $f(y_1,\dots,y_n)$ has leading term
$cy_1^{b_1}\cdots y_n^{b_n}$. Thus
\[
\deg f(y_1,\dots,y_n)=\deg y_1^{b_1}\cdots y_n^{b_n}
=\sum_{i=1}^n b_i \deg y_i>
\sum_{i=1}^n b_i=\deg f.
\]
Therefore part (b) in Definition \ref{xxdef2.1}(2) is verified.
Part (a) can be checked similarly. The assertion follows.
\end{proof}

\begin{definition}
\label{xxdef2.3} Retain the hypotheses in Definition \ref{xxdef2.1}.
Let $\tr: A\to R=F$ be the regular trace function \eqref{1.7.1}
and $w=\rk(A_R/R)$.
We say the discriminant of $A$ over $R$ is \emph{dominating}
(respectively, \emph{locally dominating})
if the discriminant ideal $D_w(A,\tr)$ is a principal ideal of $R$
generated by a dominating (respectively, locally dominating)
element.
\end{definition}

Usually we assume that $A$ is finitely generated free over $R$; then
by Proposition \ref{xxpro1.4}(3), $D_w(A,\tr)$ is generated by
$d(A/R)$. In this case we also say that $d(A/R)$
is dominating in Definition \ref{xxdef2.3}.
We now recall a few other definitions given in the introduction.

\begin{definition}
\label{xxdef2.4}
Let $\Af$ be the category consisting of all $k$-flat $k$-algebras $A$
satisfying the following conditions:
\begin{enumerate}
\item
$A$ is a filtered algebra as in Definition \ref{xxdef2.1}
such that the associated graded ring $\gr A$ is a connected graded domain,
\item
$A$ is a finitely generated free module over its center $R$, and
\item
the discriminant $d(A/R)$ is dominating.
\end{enumerate}
The morphisms in this category are isomorphisms of algebras.
\end{definition}

\begin{definition}
\label{xxdef2.5} Let $(A,Y)$ be defined as in Definition
\ref{xxdef2.1}.
\begin{enumerate}
\item
An algebra automorphism $g$ of $A$ is
said to be \emph{affine} if $\deg g(x_i)=1$ for all $i$, or
equivalently, $g(Y)\subset Y\oplus k$.
\item
If every $g\in \Aut(A)$ is affine, we call $\Aut(A)$ \emph{affine}.
\end{enumerate}
\end{definition}

The definition of an affine automorphism (and that of a dominating
element) is dependent on $Y$ (or the filtration of $A$). But in most
cases, the filtration (which is not unique in general) is relatively
easy to determine. Dominating elements help us to determine the
automorphism group in the following way.

\begin{lemma}
\label{xxlem2.6} Let $A$ be an algebra generated by $Y$ with a
locally dominating element $f$. If $g\in \Aut(A)$ such that $g(f)=\lambda f$
for some $0\neq \lambda \in k$, then $g$ is affine.
\end{lemma}

\begin{proof} Since $g$ is an automorphism, the elements $g_i:=g(x_i)$
are not in $k$.  Thus $\deg g_i\geq 1$. If $\deg g_{i_0}>1$ for some
$i_0$, then $\deg f(g_1,\dots,g_n)>\deg f$ as $f$ is locally dominating. Note
that $g(f)=f(g_1,\dots,g_n)$, whence $\deg g(f)>\deg f$, contradicting
the hypothesis $g(f)=\lambda f$. Therefore $\deg g(x_i)=1$ for all
$i$.
\end{proof}

By Lemma \ref{xxlem1.8}(6), the discriminant $d(B/R)$ is
$g$-invariant for any automorphism $g$ such that $g$ and $g^{-1}$
preserve $R$. In several situations -- see Theorem~\ref{xxthm4.9}(1),
Example \ref{xxex5.1}, and \cite{CPWZ2} -- we show
that the discriminant is dominating, and so any automorphism $g$
is affine by Lemma \ref{xxlem2.6}. Here is a general statement,
which is also Theorem \ref{xxthm0.3}(1).

\begin{theorem}
\label{xxthm2.7} Let $A$ be a filtered algebra with standard filtration
$F_n A=(Y\oplus k)^n$. Assume that the discriminant of $A$
over its center $R$ is locally dominating in $A$ \textup{(}for example,
$A$ is in $\Af$\textup{)}.
Then every automorphism  of $A$ is affine.
\end{theorem}

\begin{proof}
This follows from Lemmas \ref{xxlem1.8}(6) and \ref{xxlem2.6}.
\end{proof}

\begin{remark}\label{xxrem2.8}
For a filtered algebra $A$ generated by $Y=\bigoplus_{i=1}^n kx_i$, here is
a general way of determining affine automorphisms of $A$. For simplicity,
let $k$ be a field. Write
\[
g(x_i)=\sum_{j=1}^n a_{ij} x_j+ b_i, \quad {\text{for all $i=1,\dots,n$}},
\]
with $(a_{ij})_{n\times n}\in GL_n(k)$ and $b_i\in k$. Write the
inverse of $g$ on the generators as
\[
g^{-1}(x_i)=\sum_{j=1}^n a'_{ij} x_j+ b'_i,
\quad {\text{for all $i=1,\dots,n$}},
\]
with
$(a'_{ij})_{n\times n}=(a_{ij})^{-1}\in GL_n(k)$ and $b'_i\in k$.
List all of the relations of $A$, say,
\[
r_s(x_1,\dots,x_n)=0
\]
for $s=1,2,\dots$.
Then $g$ is an automorphism of $A$ if and only if
\[
r_s(g(x_1),\dots,g(x_n))=r_s(g^{-1}(x_1),\dots,g^{-1}(x_n))=0
\]
for all $s$. After we fix a $k$-basis of $A$,
this is an explicit linear algebra problem and can be
solved completely if we have an explicit description of the relations
$r_s$. If $A$ is noetherian, then it is enough to use
$r_s(g(x_1),\dots,g(x_n))=0$ only. In conclusion, in many situations
it is relatively easy to determine all affine automorphisms of
$A$.

Let $\Aut_{\af}(A)$ be the set of affine
automorphisms of $A$. Since $k$ is a field,
$\Aut_{\af}(A)$ is a subgroup of
$GL(Y\oplus k)$. Since every relation of $A$ gives rise to some
closed conditions, $\Aut_{\af}(A)$ is a closed subgroup of
$GL(Y\oplus k)$. As a consequence, $\Aut_{\af}(A)$ is an
algebraic group and acts on $Y\oplus k$ rationally.
\end{remark}

\section{Consequences}
\label{xxsec3}

In the previous section, we proved Theorem~\ref{xxthm0.3}(1); our goal
now is to prove the rest of that theorem. This involves an examination
of triangular automorphisms and locally nilpotent derivations.

First we consider the automorphism group of
$A[t]$ when $A$ has a dominating discriminant over
its center $R$.
For any  $g\in \Aut(A)$, $c\in k^\times$ and $r\in R$, the map
\begin{equation}\label{3.0.1}\tag{3.0.1}
 \sigma: t\to ct+r, \quad  x\to g(x), \quad {\text{for all $x\in A$}}
\end{equation}
determines uniquely a so-called \emph{triangular} automorphism of $A[t]$.
The automorphisms given in Example \ref{xxex5.12} can be viewed as triangular
automorphisms of the Ore extension $D[x_n;\tau,\delta]$ where $D$ is the
subalgebra generated by $\{x_1,\dots,x_{n-1}\}$.

One may associate the triangular automorphism
$\sigma$ \eqref{3.0.1} with the upper triangular matrix
$\begin{pmatrix} g & r\\0& c\end{pmatrix}$. The product of
two such automorphisms (or two such matrices) is given by
\[
\begin{pmatrix} g_1 & r_1\\0& c_1\end{pmatrix}
\circ \begin{pmatrix} g_2 & r_2\\0& c_2\end{pmatrix}=
\begin{pmatrix} g_1g_2 & g_1(r_2)+r_1 c_2\\0& c_1c_2\end{pmatrix}.
\]
The inverse is given by
\[
\begin{pmatrix} g & r\\0& c\end{pmatrix}^{-1}=
\begin{pmatrix} g^{-1} & -c^{-1} g^{-1}(r)\\0& c^{-1}\end{pmatrix}.
\]
This shows that all triangular automorphisms form a subgroup of
$\Aut(A[t])$, which is denoted by
\[
\Aut_{\triangular}(A[t])\quad {\text{or}} \quad
\begin{pmatrix} \Aut(A)& R\\ 0& k^\times\end{pmatrix}.
\]
Using the dominating discriminant we can show that
$\Aut_{\triangular}(A[t])=\Aut(A[t])$. The following lemma is obvious.

\begin{lemma}\label{xxlem3.1} Suppose $A$ is a finitely
generated free module over its center $R$. Let $C$ be a
commutative algebra that is $k$-flat. Then $d(A\otimes C/R\otimes C)
=_{(R\otimes C)^{\times}} d(A/R)$. If, further,
$(R\otimes C)^\times=R^{\times}$, then $d(A\otimes C/R\otimes C)
=_{R^{\times}} d(A/R)$.
\end{lemma}

The next lemma says that discriminant of $d(A[t]/R[t])$ is dominating
among the elements in $g(A)$, for $g\in \Aut(A[t])$: it controls the
degree of $g(x_i)$ for $x_i\in Y$ and for $g\in \Aut(A[t])$.  However,
it does not control the degree of $g(t)$.

\begin{lemma}\label{xxlem3.2}
Let $A$ be in $\Af$. Then the following hold.
\begin{enumerate}
\item
Let $C$ be a $k$-flat commutative filtered algebra
such that $\gr A\otimes \gr C$ is a connected graded
domain. If $g\in \Aut(A\otimes C)$, then $g(Y)\subset Y\oplus k$.
\item
Let $m$ be  a positive integer.
If $g$ is an automorphism of $A[t_1,\dots,t_m]$, then
$g(Y)\subseteq Y\oplus k$.
\end{enumerate}
\end{lemma}

\begin{proof} (2) is a consequence of (1). So we only prove
(1).

Let  $T$ be the corresponding filtered algebra $A\otimes C$
such that $\gr T=\gr A\otimes \gr C$, which is a domain by
hypothesis. Hence \eqref{2.0.1} holds and
$(A\otimes C)^\times =k^\times$. It is clear that the center of
$A\otimes C$ is $R\otimes C$.
By Lemma \ref{xxlem3.1}, $f := d(A\otimes C/R\otimes C)=_{k^\times} d(A/R)$.
Let $Y=\bigoplus_{i=1}^n kx_i$.

Consider a new filtration on the testing algebra $A\otimes C$
with assignment $\deg'(c)=2\deg (c)$ for all $c\in C$ and
$\deg'(x_i)=1$ for all $i$.  Consequently, $\deg' (c)\geq 2$
for any $c\in C\setminus k$. It is easy to verify that
$\gr' (A\otimes C)\cong (\gr A)\otimes (\gr' C)$, and the latter is
isomorphic to $(\gr A)\otimes (\gr C)$ as ungraded algebras.

Let $g\in \Aut(A\otimes C)$. Since $g$ preserves $f$ (up to a scalar),
$\deg' g(f)=\deg' f$. Since $x_i\in Y\setminus \{0\}$ are not in the center,
$y_i:=g(x_i)$ is not in the center of $A\otimes C$ for all $i$.
Consequently, $\deg y_i\geq 1$ for all $i$.
Since $f$ is dominating, there is a presentations of $f$, say
$f(x_1,\dots,x_n)$, such that
\[
\deg' g(f)=\deg' f(y_1,\dots,y_n)>\deg' f (=\deg f)
\]
if $\deg' y_i>1$ for some $i$. This yields a contradiction
and therefore $\deg' y_i\leq 1$ for all $i$.  This
means that $g(x_i)\in Y\oplus k$ for all $i$ as
$\deg' (c)\geq 2$ for any $c\in C\setminus k$.
\end{proof}

Derivations are closely related to automorphisms.
Recall that a $k$-linear map $\partial: A\to A$ is called a
\emph{derivation} if
\[
\partial (xy)=\partial(x)y+x\partial(y)
\]
for all $x,y\in A$. We call $\partial$ \emph{locally nilpotent}
if for every $x\in A$, $\partial^n(x)=0$ for some $n$. Given a
locally nilpotent derivation $\partial$ (and assuming that
${\mathbb Q}\subseteq k$), the exponential map $\exp(\partial):
A\to A$ is defined by
\[
\exp(\partial)(x)=\sum_{i=0}^{\infty} \frac{1}{i!} \partial^i(x),
\quad {\text{for all $x\in A$.}}
\]
Since $\partial$ is locally nilpotent, $\exp(\partial)$ is an
algebra automorphism of $A$ with inverse $\exp(-\partial)$.

\begin{lemma}\label{xxlem3.3} Suppose that ${\mathbb Q}\subseteq
 k$. Let $C$ be a commutative algebra that is $k$-flat.
 \begin{enumerate}
  \item
  If every $k$-algebra automorphism of $A\otimes C[t]$ restricts
  to an algebra automorphism of $A$, then every locally nilpotent
  derivation of $A\otimes C$ becomes zero when restricted to $A$.
  \item
  If $\Aut(A[t])=\Aut_{\triangular}(A[t])$, then every locally nilpotent
  derivation of $A$ is zero.
  \item
  If $A$ is in $\Af$, then every locally nilpotent
  derivation of $A[t_1,\dots,t_m]$ becomes zero when restricted to $A$.
 \end{enumerate}
\end{lemma}

\begin{proof}
 (1) Let $\partial$ be a locally nilpotent derivation of $A\otimes C$.
Extend $\partial$ to $\partial': A\otimes C[t]\to A\otimes C[t]$ by
defining $\partial'(t)=0$ and $\partial'\mid_{A\otimes C}=\partial$.
Then $\partial'$ is a locally nilpotent derivation of $A\otimes C[t]$.
Further, $t\partial'$ is a locally nilpotent derivation of $A\otimes C[t]$.
Then the exponential map $\exp(t\partial')$ is a $k$-algebra
automorphism of $A\otimes C[t]$. By hypothesis, the restriction of
$\exp(t\partial')$ to $A$ is an automorphism of $A$. But,
\[\exp(t\partial')(x)=\sum_{i=0}^\infty \frac{t^i}{i!} \partial^i(x),
\quad {\text{for all $x\in A$}},\]
which is in $A$ only if $\partial(x)=0$. The assertion follows.

(2) This is a special case of (1) when $C=k$.

(3) Let $C=k[t_1,\dots,t_m]$. By Lemma \ref{xxlem3.2}(1) (for
$C=k[t_1,\dots,t_m,t]$),
the hypotheses of part (1) hold. Then the assertion follows from part (1).
\end{proof}

From now until Lemma \ref{xxlem3.6} we suppose that
$k$ is a field of characteristic zero.
We refer to \cite{Hu} for basic definitions about (affine) algebraic groups.
By Remark \ref{xxrem2.8},  if $\Aut(A)$ is affine, then it
is an algebraic subgroup of $GL(Y\oplus k)$.
Let $\Aut^1(A)$ denote the identity component of $\Aut(A)$,
which is the unique closed, connected, normal subgroup of finite index
in $\Aut(A)$. An element $\sigma\in \Aut^1(A)$ or in $\Aut(A)$
is called \emph{unipotent} if $Id-\sigma$, as a linear map of $Y\oplus k$,
is nilpotent.

\begin{lemma}\label{xxlem3.4}
Let $k$ be a field of characteristic zero.
Assume that $\Aut(A)$ is affine
\textup{(}namely, $\Aut(A)\subset GL(Y\oplus k)$\textup{)} and
that every locally nilpotent derivation of $A$ is zero. Then
$\Aut^1(A)$ is a torus -- it is isomorphic to $(k^\times)^r$
for some $r\geq 0$ --
and $\Aut(A)$ is an algebraic group that fits into an exact
sequence
\[
1\to (k^\times)^r\to \Aut(A)\to S\to 1
\]
for some finite group $S$.
\end{lemma}

\begin{proof}
 Let $\sigma$ be in $\Aut(A)$ such that $Id-\sigma$ is
nilpotent on $Y\oplus k$.
Then $\log \sigma:=\sum_{n=1}^{\infty} \frac{-1}{n} (Id-\sigma)^n$
 is a locally nilpotent derivation. By hypothesis, $\log \sigma$
 is zero. Then $Id-\sigma$ is zero, so $\sigma=Id$. So every unipotent
 element in $\Aut(A)$ is the identity. Then $\Aut^1(A)$ is a torus by
\cite[Exer. 21.4.2]{Hu}. Since $\Aut^1(A)$ has finite index in $\Aut(A)$,
the exact sequence is clear.
\end{proof}

Now we are ready to prove Theorem \ref{xxthm0.3}(2,3,4).

\begin{theorem}
\label{xxthm3.5} Let $k$ be a field of characteristic zero
and $A$ be in $\Af$. Then the following hold.
\begin{enumerate}
\item
$\Aut(A[t])=\Aut_{\triangular}(A[t]).$
\item
Every locally nilpotent derivation $\partial$ of $A[t]$ is of the form
\[\partial(x)=0 \quad {\text{for all $x\in A$}},
\quad \partial(t)=r \quad {\text{for some $r\in R$}}.\]
\item
Every locally nilpotent derivation of $A$ is zero.
\item
$\Aut(A)$ is an algebraic group that fits into an exact
sequence
\[
1\to (k^\times)^r\to \Aut(A)\to S\to 1
\]
for some finite group $S$.
 \end{enumerate}
\end{theorem}

\begin{proof} (1)
Let $Y=\bigoplus_{i=1}^n kx_i$ and $g\in \Aut(A[t])$. By Lemma
\ref{xxlem3.2}(2), $g(x_i)\in
Y\oplus k\subset A$, or $g(A)\subset A$.
Applying Lemma \ref{xxlem3.2}(2)
to $h:=g^{-1}$, we have $h(A)\subset A$. Thus
$g|_{A}$ and $h|_{A}$ are inverse to each other and
hence $g|_A\in \Aut(A)$. Let $g(t)=\sum_{i=0}^n a_i t^i$
with $a_n\neq 0$ and $h(t)=\sum_{j=0}^m b_j t^j$
with $b_m\neq 0$. Then $gh(t)=\sum_{i=0}^{nm} c_i t^i$
with $c_{nm}=a_n (b_m)^n\neq 0$. Since $gh(t)=t$,
$nm=1$ (consequently, $n=m=1$) and $a_1 b_1=1$.
Thus $c:=a_1\in R^\times =k^\times$. This shows that
$g(t)=ct+a_0$ where $c\in k^\times $ and $a_0\in A$.
Since $t$ is central, $r:=a_0\in R$. The assertion
follows.

(2) By Lemma \ref{xxlem3.3}(3), $\partial(x)=0$
for all $x\in A$. Let $\partial(t)=\sum_{i=0}^{d}
c_i t^i$ for some $c_i\in A$. Suppose $\partial(t)\neq 0$
and it has $t$-degree $d$ (namely, $c_d\neq 0$).
If $n>0$, the induction shows that $\partial^{n}(t)$
has $t$-degree $nd-(n-1)$. Hence $\partial$ is not
locally nilpotent, a contradiction. Thus $\partial(t)
=c_0\in A$. Since $xt=tx$ for all $x\in A$, applying
$\partial$ to the equation, we have
$xc_0=c_0x$. Thus $c_0$ is in the center of $A$ and
the assertion follows.

(3) Follows from part (1) and Lemma \ref{xxlem3.3}(2).

(4) Follows from Theorem \ref{xxthm2.7}, part (3) and Lemma \ref{xxlem3.4}.
\end{proof}

Next we compute another automorphism group and we assume that $k$ is a
commutative domain. For any positive integer $m$, define
$A[\underline{t}_m^{\pm 1}]$ to be the Laurent polynomial extension
$A[t_1^{\pm 1}, t_2^{\pm 1}, \cdots, t_m^{\pm1}]$.  The following
lemma is easy and the proof is omitted.

\begin{lemma} \label{xxlem3.6} Let $A$ be any algebra.
 \begin{enumerate}
  \item
  $(A[\underline{t}_m^{\pm 1}])^\times=\bigcup_{(n_s)\in {\mathbb Z}^m}
  A^\times \cdot t_1^{n_1}t_2^{n_2}\cdots t_m^{n_m}$.
  \item
  Suppose $A^\times =k^\times$. Then every automorphism of
  $A[\underline{t}_m^{\pm 1}]$ preserves $k[\underline{t}_m^{\pm 1}]$.
  \item
  $\Aut(k[\underline{t}_m^{\pm 1}])=(k^\times)^{m} \rtimes GL_m({\mathbb Z})$.
 \end{enumerate}
\end{lemma}

\begin{proposition}
\label{xxpro3.7}
Let $m$ be a positive integer.
If $A^\times=k^\times$, then
\[
\Aut(A[\underline{t}_m^{\pm 1}])=
\Aut_{k[\underline{t}_m^{\pm 1}]}(A[\underline{t}_m^{\pm 1}])\times
\Aut(k[\underline{t}_m^{\pm 1}]).
\]
\end{proposition}

\begin{proof}
 Let $g\in \Aut(A[\underline{t}_m^{\pm 1}])$. By Lemma \ref{xxlem3.6}(2),
 $g\mid_{k[\underline{t}_m^{\pm 1}]}:=g_2$ preserves
$k[\underline{t}_m^{\pm 1}]$.
 Thus $g_2\in \Aut(k[\underline{t}_m^{\pm 1}])$. Then $(1\otimes g_2)^{-1}
 g$ is in $\Aut_{k[\underline{t}_m^{\pm 1}]}(A[\underline{t}_m^{\pm 1}])$.
 The assertion holds.
\end{proof}

If $\gr A$ is a connected graded domain, then $A^\times =k^\times$.
Therefore Proposition \ref{xxpro3.7} applies.  Note that
 $\Aut_{k[\underline{t}_m^{\pm 1}]}(A[\underline{t}_m^{\pm 1}])$
 is affine  by Lemma \ref{xxlem3.2}(1), and therefore computable
[Remark \ref{xxrem2.8}].
 By using Proposition \ref{xxpro3.7}, $\Aut(A[\underline{t}_m^{\pm 1}])$
 can be described explicitly. In general, it would be interesting to
 understand the relationship between $\Aut(A\otimes C)$ and
 the pair $(\Aut(A),\Aut(C))$. Under the situation of Lemma
 \ref{xxlem3.3}(1), we have some useful information. On the other
 hand, this relationship is extremely complicated when $A$ and $C$
 are arbitrary.

 To conclude this section we give two examples. The first one  shows that
 parts (2,3,4) of Theorem \ref{xxthm0.3} do not
 follow from part (1) of Theorem \ref{xxthm0.3}, and the second one shows
that Theorem \ref{xxthm0.3}(3) fails without the hypothesis that
$\ch k=0$.

 \begin{example}
  \label{xxex3.8}
  Let $q\in k^\times$ be not a root of unity. Let $A$ be
  the skew polynomial ring generated by $x_1,x_2,x_3$
  subject to the relations
  \[
     x_2x_1=x_1x_2, \quad
     x_3x_1=qx_1x_3, \quad
     x_3x_2=qx_2x_3.
  \]
  Let $Y=kx_1\oplus kx_2\oplus kx_3$. Then $A$ is graded
  with $F_1 A=Y\oplus k$.
Using the fact that $q$ is not a root of unity, one can check that
every automorphism $g$ of $A$ is affine, namely, $g(Y)\subset Y$.
In fact, $\Aut(A)\cong GL(2,k)\times k^\times$. So it is not of the
form in Theorem \ref{xxthm0.3}(4). The map $\partial: x_1\to 0,
x_2\to x_1, x_3\to 0$ extends to a nonzero locally nilpotent
derivation. Further, there is
an automorphism of $A[t]$
\[
h: x_1\to x_1, x_2\to x_2+tx_1, x_3\to x_3, t\to t+a,
\]
which is not in $\Aut_{\triangular}(A[t])$. Therefore parts (2,3,4) of
Theorem \ref{xxthm0.3} fail.
 \end{example}

\begin{example}
\label{xxex3.9} Let $A$ be the skew polynomial ring $k_{-1}[x_1,x_2]$
and $R:=k[x_1^2,x_2^2]$ be the center of $A$.
For any $a,b\in k$ and any $h\in R$, define a derivation
by
\[
\partial: x_1\to a x_1 h, \quad x_2\to b x_2 h.
\]
This $\partial$ extends to a derivation for any commutative
base ring $k$ and, by induction, $\partial(x_1^m x_2^n)
=(am+bn) x_1^mx_2^n h$ for all non-negative integers $m$ and $n$.

Now assume that $\ch k=p>2$. Let $a=1$, $b=0$ and $h=x_1^2$.
Then $\partial(x_2)=0$ and $\partial(x_1^m)=m x_1^{m+2}$. By
induction, $\partial^n(x_1)=1\cdot 3 \cdot 5  \cdots  (2n-1)\;
x_1^{2n+1}$ for all $n\geq 1$. It
follows that $\partial^p=0$. Therefore $\partial$ is
locally nilpotent. By Example \ref{xxex1.7}(2),
the discriminant of $A$ over its center is
$x_1^4 x_2^4$, which is dominating. So Theorem \ref{xxthm0.3}(3)
fails without the hypothesis that $\ch k=0$.
Let $d$ be the discriminant $x_1^4x_2^4$. Then $\partial(d)=
4 x_1^6 x_2^4=4 d x_1^2\neq 0$.  In this case, $d$ is not an eigenvector
of $\partial$.
\end{example}

\section{An example}
\label{xxsec4}

In this section, we assume that $k$ is a commutative domain
and that $2$ is invertible in $k$.
Our goal here is to prove Theorem~\ref{xxthm0.1} by computing enough
information about the discriminant for the algebra $W_n$ to show that
this algebra is in $\Af$.

Let $\mathcal{A}
:=\{a_{ij}\mid 1\leq i<j \leq n\}$ be a set of scalars in $k$.
Define the $(-1)$-quantum Weyl algebra $V_n(\mathcal{A})$ to be
generated by $\{x_1,x_2,\dots,x_n\}$ subject to the relations
\[
x_i x_j+x_jx_i=a_{ij}
\]
for all $i<j$. Example \ref{xxex1.7}(1) is a special case with $n=2$ and
$a_{12}=1$. If $a_{ij}=0$ for all $i<j$, then this algebra is denoted
by $k_{-1}[x_1,\dots,x_n]$. If $a_{ij}=1$ for all $i<j$, we get the
algebra $W_n$ of the introduction.

We refer to \cite{MR} for the definition of global dimension,
Gelfand-Kirillov dimension (or GK-dimension, for short), and
Krull dimension.

\begin{lemma}
\label{xxlem4.1} The following hold for $V:=V_n(\mathcal{A})$.
\begin{enumerate}
\item
$V$ is an iterated Ore extension
$k[x_1][x_2;\sigma_2,\delta_2]\cdots [x_n;\sigma_n,\delta_n]$ where
$\sigma_j: x_i\mapsto -x_i$ and $\delta_j: x_i\mapsto a_{ij}$ for all $i<j$.
\item
$V$ is a filtered algebra with associated graded ring $\gr V\cong
k_{-1}[x_1,\dots,x_n]$.
\item
If $k$ is a field, then
$V$ is a noetherian Auslander regular Cohen-Macaulay domain of
global dimension, GK-dimension, and Krull dimension $n$.
\item
The center of $k_{-1}[x_1,\dots,x_n]$ is
\[
\begin{cases} k[x_1^2,\dots,x_n^2]
& {\text{ if $n$ is even}},\\
k[x_1^2,\dots,x_n^2, \prod_i x_i] & {\text{ if $n$ is odd}}.
\end{cases}
\]
\item
If $n$ is even, the center of $V$ is $R:=k[x_1^2,\dots,x_n^2]$, and
$V$ is finitely generated free over $R$ of rank $2^n$.
\end{enumerate}
\end{lemma}

\begin{proof} (1) It is easy to check that $\sigma_{j+1}$ is an algebra
automorphism of $K_j:=k[x_1][x_2;\sigma_2,\delta_2]\cdots
[x_j;\sigma_j,\delta_j]$ and $\delta_{j+1}$ is a
$\sigma_{j+1}$-derivation of $K_j$. The assertion follows.

(2) Let $Y=\sum_{i=1}^n kx_i$. Then $F_n:= (k+Y)^n$ defines a
filtration of $V$ such that $\gr V$ is generated by
$\{x_1,\dots,x_n\}$ and subject to the relations $x_ix_j+x_jx_i=0$
for all $i\neq j$. The assertion follows.

(3) It is well-known that $k_{-1}[x_1,\dots,x_n]$ is a noetherian
Auslander regular Cohen-Macaulay domain of GK-dimension, Krull
dimension and global dimension $n$. Hence $V$ is a noetherian Auslander
regular Cohen-Macaulay domain of  GK-dimension and global dimension
$n$ and Krull dimension at most $n$. Since $V$ is PI, the Krull
dimension is equal to its GK-dimension.

(4) Since $k_{-1}[x_1,\dots,x_n]$ is $\Z^n$-graded and
$\Z^n$ is an ordered group, the center of
$k_{-1}[x_1,\dots,x_n]$ is $\Z^n$-graded. So every central
element is a linear combination of monomials. It can be checked
directly that each central monomial is generated by
$x_1^2,\dots,x_n^2$ when $n$ is even and by $x_1^2,\dots,x_n^2,
\prod_i x_i$ when $n$ is odd.

(5) Let $C$ be the center of $V$. Since $x_i x_j^2-x_j^2 x_i=
(-x_jx_i+a_{ij})x_j-x_j(-x_ix_j+a_{ij})=0$, $x_j^2\in C$. Thus
$k[x_1^2,\dots,x_n^2]\subset C$. It is clear that $\gr C\subset
C(\gr V)=k[x_1^2,\dots,x_n^2]$. Thus $\gr C=k[x_1^2,\dots,x_n^2]$.
By lifting, $C=k[x_1^2,\dots,x_n^2]$.
\end{proof}

We are interested in $\Aut(V)$, which is related to the graded
algebra automorphism group, denoted by $\Aut_{\gr}$, of
$k_{-1}[x_1,\dots,x_n]$. Let $[n]$ denote the set
$\{1,2,\dots,n\}$ and $S_n$ be the symmetric group consisting of
all permutations of $[n]$. Recall that $W_n$ is the algebra
$V(\{1\}_{i<j})$, namely, $a_{ij}=1$ for all $1\leq i<j\leq n$.

\begin{lemma}
\label{xxlem4.2} The following hold.
\begin{enumerate}
\item \cite[Lemma 1.12]{KKZ}
$\Aut_{\gr}(k_{-1}[x_1,\dots,x_n])= S_n\ltimes
(k^\times )^n$.
\item
$S_n\times \{\pm 1\}\subseteq \Aut(W_n)$.
\end{enumerate}
\end{lemma}

\begin{proof}
(2) is clear. We only prove (1). This was
proved in \cite[Lemma 1.12]{KKZ} when $k$ is a field.
The assertion in the general case follows by passing
from $k$ to the ring of fractions of $k$.
\end{proof}

Here is an application of Remark \ref{xxrem2.8}. Recall that
$\Aut_{\af}(V)$ denotes the group of affine automorphisms of $V$.
We take $Y=\bigoplus_{i=1}^n kx_i$ for the algebra $V$.

\begin{lemma}
\label{xxlem4.3} Let $g$ be an affine automorphism of $V$. Then there
is a permutation $\sigma \in S_n$ and $r_i\in k^\times$ such that
$g(x_i)=r_i x_{\sigma(i)}$ for all $i$. As a consequence,
\[
\Aut_{\af}(W_n)=\begin{cases}
S_2 \ltimes k^\times & {\text{if $n=2$}},\\
S_n \times \{\pm 1\} & {\text{if $n\geq 3$.}}
\end{cases}
\]
\end{lemma}

\begin{proof} Since $g$ preserves the filtration, the associated
graded automorphism, denoted by $\bar{g}$, is a graded algebra
automorphism of $k_{-1}[x_1,\dots,x_n]$. By Lemma
\ref{xxlem4.2}(1), there is a permutation $\sigma \in S_n$ and
$r_i\in k^\times$ such that $\bar{g}(x_i)=r_i x_{\sigma(i)}$ for all
$i$. Thus we have $g(x_i)=r_i x_{\sigma(i)}+a_i$ for some $a_i\in
k$. It remains to show that $a_i=0$ for all $i$. Applying $g$ to the
relations $x_ix_j+x_jx_i=a_{ij}$, we have
\begin{align*}
a_{ij}&= g(x_ix_j+x_jx_i)\\
&=(r_i x_{\sigma(i)}+a_i)(r_j x_{\sigma(j)}+a_j) + (r_j
x_{\sigma(j)}+a_j)(r_i x_{\sigma(i)}+a_i)\\
&=r_ir_j(x_{\sigma(i)}x_{\sigma(j)}+x_{\sigma(j)}x_{\sigma(i)})
+2a_ir_j x_{\sigma(j)}+2a_jr_i x_{\sigma(i)}+2a_ia_j\\
&= r_ir_j a_{\sigma(i)\sigma(j)}+2a_ir_j x_{\sigma(j)} +2a_jr_i
x_{\sigma(i)}+2a_ia_j.
\end{align*}
Since $r_i\neq 0$, we have $a_j=0$ for all $j$. The consequence
follows easily from the fact that in $W_{n}$, we have
$a_{ij}=1$ for all $i<j$, and so
\[
1=a_{ij}= r_ir_j
a_{\sigma(i)\sigma(j)}=r_ir_j
\]
for all $i<j$.
\end{proof}

Let $I=\{i_1,i_2,\dots,i_s\}$ be a set of integers between $1$ and
$n$ with repetitions. We let $X_I=x_{i'_1}x_{i'_2}\cdots x_{i'_s}\in
V_n(\mathcal{A})$ where $\{i'_1,i'_2,\dots, i'_s\}$ is a
re-ordering of the elements in $I$ such that $i'_1\leq i'_2\leq \cdots
\leq i'_s$. Since $V_n(\mathcal{A})$ has a PBW basis,
$V_n(\mathcal{A})$ has a $k$-linear basis consisting of
all different monomials $X_I$. For two sets $I$ and $J$ of integers
between $1$ and $n$, let $I+J$ denote the union of $I$ and $J$ with
repetitions. Suppose $K_1$ and $K_2$ are two sets of integers.  We
write $K_1\to K_2$ if there is a presentation $K_1=\{k_1,\dots,
k_w\}$ and $K_2=\{k'_1,\dots, k'_w\}$ such that $k_\alpha> k'_\alpha$
for all $\alpha$ from $1$ to $w$.

\begin{lemma}
\label{xxlem4.4} Let $I=\{i_1, i_2,\dots, i_s\}$ and
$J=\{j_1,j_2,\dots, j_u\}$ where the $i$'s and $j$'s are in
non-decreasing order. Then
\[
X_I X_J=c X_{I+ J}+\sum_{\substack{\emptyset \neq K_1\subset I \\
\emptyset \neq K_2\subset J \\ K_1\to K_2}} c_{K_1,K_2} X_{(I \setminus
K_1)+ (J\setminus K_2)}
\]
where $c\in k^\times, c_{K_1,K_2}\in k$.
\end{lemma}

\begin{proof} First suppose that $I$ has a single element
$i_1$. If $i_1\leq j_1$, then
the assertion is trivial. Now assume $i_1>j_1$. By induction on $u$,
we have
\[
x_{i_1}x_{j_2}\cdots
x_{j_u}=c'X_{\{i_1\}+(J\setminus \{j_1\})}+\sum_{\substack{
K_2=\{k_1\}\subset (J\setminus \{j_1\})\\
i_1>k_1}} c_{K_2}
X_{(J\setminus (K_2+\{j_1\}))}.
\]
Then
\begin{align*}
X_I X_J&=x_{i_1} x_{j_1}\cdots x_{j_u} \\
&=(x_{j_1}x_{i_1}+a_{i_1j_1})x_{j_2}\cdots x_{j_u}\\
&=x_{j_1}x_{i_1}x_{j_2}\cdots
x_{j_u}+a_{i_1j_1}x_{j_2}\cdots x_{j_u}\\
&= x_{j_1}\left[c'X_{\{i_1\}+(J\setminus
\{j_1\})}+\sum_{\substack{K_2=\{k_1\}\subset (J\setminus
\{j_1\}) \\ i_1>k_1}} c_{K_2} X_{(J\setminus
(K_2+\{j_1\}))}\right]\\
&\qquad +a_{i_1j_1}x_{j_2}\cdots x_{j_u}\\
&=c X_{\{i_1\}+J}+\sum_{\substack{\emptyset \neq K_2\subset J \\
I\to K_2}} c_{K_2} X_{(J\setminus K_2)}.
\end{align*}
Now we assume that $|I|>1$. We write
$I=\{i_1\}+ I'$ where $|I'|=|I|-1$. By induction,
\[
X_{I'} X_J=b X_{I'+ J}+\sum_{\substack{\emptyset \neq K_1\subset I' \\
\emptyset \neq K_2\subset J \\ K_1\to K_2}} b_{K_1,K_2} X_{(I' \setminus
K_1)+ (J\setminus K_2)}.
\]
Then
\begin{align*}
X_I X_J&= x_{i_1} X_{I'}X_J
=x_{i_1} \left[b X_{I'+ J}+\sum_{\substack{\emptyset \neq K_1\subset I' \\
\emptyset \neq K_2\subset J \\ K_1\to K_2}} b_{K_1,K_2} X_{(I'
\setminus K_1)+ (J\setminus K_2)}\right]\\
&=b x_{i_1} X_{I'+ J}+\sum_{\substack{\emptyset \neq K_1\subset I' \\
\emptyset \neq K_2\subset J \\ K_1\to K_2}} b_{K_1,K_2} x_{i_1}X_{(I'
\setminus K_1)+ (J\setminus K_2)}.
\end{align*}
For $x_{i_1} X_{I'+ J}$ and $x_{i_1}X_{(I' \setminus K_1)+
(J\setminus K_2)}$, we use the case when $|I|=1$. Note that $i_1$ is
no larger than any element in $I'$. So
\begin{align*}
x_{i_1} X_{I'+ J}&=c'X_{I+J}+\sum_{\substack{K_2=\{k_1\}\subset
J \\ i_1>k_1}} c_{K_2} X_{(I+J\setminus
(K_2+\{i_1\}))}\\
&=c'X_{I+J}+\sum_{\substack{K_1=\{i_1\} \\ K_2=\{k_1\}\subset J \\
K_1\to K_2}} c_{K_2} X_{(I\setminus K_1)+(J\setminus K_2))}.
\end{align*}
Similarly, by using the fact that $i_1$ is no larger than any element
in $I'$, one can obtain that the linear combination
\[
\sum_{\substack{\emptyset
\neq K_1\subset I' \\ \emptyset \neq K_2\subset J \\
K_1\to K_2}} b_{K_1,K_2} x_{i_1}X_{(I' \setminus K_1)+ (J\setminus K_2)}
\]
is of the form
\[
\sum_{\substack{\emptyset \neq K_1\subset I \\
\emptyset \neq K_2\subset J \\ K_1\to K_2}} c_{K_1,K_2} X_{(I \setminus K_1)+
(J\setminus K_2)}.
\]
The assertion follows.
\end{proof}

For the rest of this section, we work on
computing the discriminant of $V_n(\mathcal{A})$ and
proving Theorem~\ref{xxthm0.1}.

Let $B=V=V_n(\mathcal{A})$ and $R=k[x_1^2,\dots,x_n^2]$. Then $B$ is
a finitely generated free module over $R$ of rank $2^n$ (and $R$ is the
center of $B$ if $n$ is even).  Let $\tr:
B\to R$ be the regular trace map as defined in Example
\ref{xxex1.2}(3). For any set of elements
$X=\{f_1,\dots, f_w\}$ in $V$, define
\begin{equation}\label{eqn-omega}
\Omega(X)=\Omega(f_1,\dots,f_n)=\sum_{\sigma \in S_w} (-1)^{|\sigma|} f_{\sigma(1)}\cdots
f_{\sigma(w)}.
\end{equation}
Let $x_{i_1i_2\cdots i_w}$ denote the element
$\Omega(x_{i_1},x_{i_2},\dots,x_{i_w})$.

\begin{lemma}
\label{xxlem4.5} We work in the algebra $V:=V_n(\mathcal{A})$.
\begin{enumerate}
\item
$\tr(1)=2^n$.
\item
$V$ is $\Z/(2)$-graded: $V = V_{\textup{even}}\oplus
V_{\textup{odd}}$ with $x_i$ having odd degree.
\item
If $f$ has odd degree, then $\tr(f)=0$. As a consequence,
if $w$ is odd, then $\tr(x_{i_1i_2\cdots i_w})=0$.
\item
If $w$ is even,
then $\tr(\Omega(f_1,\dots,f_w))=0$. As a consequence,
if $w$ is even, then $\tr(x_{i_1i_2\cdots i_w})=0$.
\end{enumerate}
\end{lemma}

\begin{proof} (1,2,3) These are clear.

(4) Using the trace property $\tr(ab)=\tr(ba)$, we have
\begin{align*}
\tr(\Omega(f_1,\dots,f_w))&=\tr(\sum_{\sigma \in S_w} (-1)^{|\sigma|}
f_{\sigma(1)}\cdots f_{\sigma(w)})
=\sum_{\sigma \in S_w} (-1)^{|\sigma|} \tr(f_{\sigma(1)}\cdots
f_{\sigma(w)})\\
&=\sum_{\sigma \in S_w} (-1)^{|\sigma|}
\tr(f_{\sigma(w)}f_{\sigma(1)}\cdots
f_{\sigma(w-1)})\\
&=\sum_{\sigma \in S_w} (-1)^{|\sigma (1,2,3,\dots,w)|}
\tr(f_{\sigma(1)}\cdots
f_{\sigma(w-1)}f_{\sigma(w)})\\
&=-\sum_{\sigma \in S_w} (-1)^{|\sigma|} \tr(f_{\sigma(1)}\cdots
f_{\sigma(w-1)}f_{\sigma(w)})
=-\tr(\Omega(f_1,\dots,f_w)).
\end{align*}
Since $2$ is invertible in $k$, the assertion follows.
\end{proof}

\begin{lemma}
\label{xxlem4.6} We continue to work in the algebra $V$.
\begin{enumerate}
\item
If $i_1<i_2<\cdots< i_s$ and $s>0$, then $\tr(x_{i_1}\cdots
x_{i_s})\in k$.
\item
If $I = \{i_1<i_2<\cdots< i_s \}$ and $J = \{ j_1<j_2<\cdots< j_u \}$,
then
\[
\tr(x_{i_1}\cdots x_{i_s} x_{j_1}\cdots x_{j_u})
=b x_{k_1}^2x_{k_2}^2\cdots x_{k_n}^2+ \cwlt
\]
where
$\{k_1,k_2,\dots, k_n\}=I\cap J$ and $b\in k$.
\item
If $i_1<i_2<\cdots< i_s$, then
\[
\tr(x_{i_1}\cdots x_{i_s} x_{i_1}\cdots x_{i_s})
=c x_{i_1}^2x_{i_2}^2\cdots x_{i_s}^2+ \cwlt
\]
for some $c\in
k^\times$.
\end{enumerate}
\end{lemma}

\begin{proof} (1) We compute the trace using the basis
\[
\{x_{j_1}x_{j_2}\cdots x_{j_u} \, | \, j_1<j_2<\cdots < j_u \}.
\]
Write $I = \{i_1<i_2<\cdots< i_s \}$ and $J = \{ j_1<j_2<\cdots< j_u
\}$.  We use Lemma~\ref{xxlem4.4} to compute:
\[
(x_{i_1}\cdots x_{i_s}) (x_{j_1}\cdots x_{j_u})=X_I X_J
=c X_{I+ J}+\sum_{\substack{\emptyset \neq K_1\subset I \\
\emptyset \neq K_2\subset J \\
K_1\to K_2}} c_{K_1,K_2} X_{(I \setminus K_1)+
(J\setminus K_2)}.
\]
If $X_{I+J}=r X_J$ for some $r\in R$, then $r$ is a scalar multiple
of $x_{k_1}^2\cdots x_{k_w}^2$ where $\{k_1,\dots, k_w\}=I\cap J$.
As a consequence $J=(I\setminus K)+(J\setminus K)$, which is
impossible as $|K|>0$. If $X_{(I \setminus K_1)+ (J\setminus K_2)}=r
X_J$ for some $r\in R$, then $r$ is a scalar multiple of
$x_{k_1}^2\cdots x_{k_w}^2$ where $K:=\{k_1,\dots,
k_w\}=(I\setminus K_1)\cap (J\setminus K_2)$. As a consequence
$J=(I\setminus K+K_1)+(J\setminus K+K_2)$. If $|K|>0$, then $K$ is
not in $(I\setminus K+K_1)+(J\setminus K+K_2)$, a contradiction.
Therefore, the only possible case is when $K$ is empty. When $K$ is
empty, the coefficient of $X_J$ is in $k$. The assertion follows.

(2) By Lemma \ref{xxlem4.4}, we need to compute
\[
\tr(X_I X_J)
=c \tr(X_{I+ J})+\sum_{\substack{\emptyset \neq K_1\subset I\\
\emptyset \neq K_2\subset J \\ K_1\to K_2}} c_{K_1,K_2} \tr(X_{(I
\setminus K_1)+ (J\setminus K_2)}).
\]
Let $I\cap J=K=\{k_1, \dots, k_n \}$. Clearly
\[
\tr(X_{I+ J})=x_{k_1}^2\cdots x_{k_n}^2 \tr(X_{(I\setminus
K)+(J\setminus K)})=x_{k_1}^2\cdots x_{k_n}^2 b
\]
for some $b=\tr(X_{(I\setminus
K)+(J\setminus K)})\in k$
by part (1).

Let $(I\setminus K_1)\cap (J\setminus K_2) = K' =
\{k'_1, \dots, k'_m \}$. Clearly
\[
\tr(X_{(I \setminus K_1)+ (J\setminus K_2)})=x_{k'_1}^2\cdots
x_{k'_m}^2 \tr(X_{(I\setminus (K'\cup K_1))+(J\setminus (K'\cup
K_2))})=x_{k'_1}^2\cdots x_{k'_m}^2 b'
\]
for some $b'\in k$ by part (1). Since $K'$ is a
subset of $K$, $\tr(X_{(I \setminus K_1)+ (J\setminus K_2)})$ is
either a scalar multiple of $\tr(X_{I+ J})$ or a scalar multiple of
some monomial in $\cwlt$. The assertion follows.

(3) For the most part, this is a special case of part (2). To
prove $c$ is invertible, we note $\tr(x_{i_1}^2\cdots
x_{i_s}^2)=2^nx_{i_1}^2\cdots x_{i_s}^2$ and that $2$ is invertible.
\end{proof}

\begin{remark}
\label{xxrem4.7} Let $V=W_n$, so the relations are $x_ix_j+x_jx_i=1$
for all $i\neq j$. Then we have an explicit formula for the trace of
each basis element $x_{i_1}\cdots x_{i_s}$, where $1 \leq i_{1} <
\dots < i_{s} \leq n$.
\begin{enumerate}
\item
If $s$ is odd, then $\tr(x_{i_1}\cdots x_{i_s})=0$, by Lemma
\ref{xxlem4.5}(3).
\item
If $\sigma\in S_n$ is a permutation of $[n]$, then, by
Lemmas \ref{xxlem1.8}(2) and \ref{xxlem4.6}(1), $\tr(x_{i_1}\cdots
x_{i_s})=\tr(x_{\sigma(i_1)}\cdots x_{\sigma(i_s)})$.
\item
If $s$ is even, then $\tr(x_{i_1}\cdots x_{i_s})=2^{n-s/2}$.
To see this, we use induction on $s$. Note that
$\tr(x_{i_1}x_{i_2}x_{i_3}\cdots
x_{i_s})=\tr(x_{i_2}x_{i_1}x_{i_3}\cdots x_{i_s})$ by part (2). Using
the relation, we have
\begin{align*}
\tr(x_{i_3}\cdots
x_{i_s})&=\tr((x_{i_1}x_{i_2}+x_{i_2}x_{i_1})x_{i_3}\cdots
x_{i_s})\\
&=\tr(x_{i_1}x_{i_2}x_{i_3}\cdots
x_{i_s})+\tr(x_{i_2}x_{i_1}x_{i_3}\cdots x_{i_s})\\
&=2\tr(x_{i_1}x_{i_2}x_{i_3}\cdots x_{i_s}).
\end{align*}
\end{enumerate}
\end{remark}

For any nonzero element $f$ in the (graded) polynomial ring
$k[x_1^2,x_2^2,\dots,x_n^2]$, let $\pr(f)$ denote the highest degree
component of $f$, which is called the \emph{principal term} of $f$ or
the \emph{leading term} of $f$.

Using the basis
\[
\left\{X_I=x_{i_1}\cdots x_{i_s}\mid I = \{i_{1} < \dots
< i_{s} \}\subset [n]\right\}
\]
to compute the discriminant, we need to compute the determinant of the
matrix
\[
M=(m_{IJ})_{2^n\times 2^n},
\]
where $m_{IJ}=\tr(x_{i_1}\cdots x_{i_w}x_{j_1}\cdots x_{j_s})$. By Lemma
\ref{xxlem4.6}, we have the following.
\begin{itemize}
\item
$m_{\emptyset,\emptyset}=2^n$,
\item
if $I=\{i_1,\dots, i_s\}$, then $\pr(m_{II})$ is of the form $c
x_{i_1}^2\cdots x_{i_s}^2$ where $c\in k^\times$, and other terms of
$m_{II}$ are cwlt $x_{i_1}^2\cdots x_{i_s}^2$,
\item
for every pair $I\neq J$, $m_{IJ}$ is cwlt both $\pr(m_{II})$ and
$\pr(m_{JJ})$.
\end{itemize}

Therefore we have the following.

\begin{proposition}
\label{xxpro4.8} Retain the notation above.
\begin{enumerate}
\item
The product $\prod_{I\subset [n]} m_{II}$ has principal term of
the form $c(\prod_{i=1}^n x_i^2)^{2^{n-1}}$ for some $c\in
k^\times$.
\item
Thus $\prod_{I\subset [n]} m_{II}=c(\prod_{i=1}^n
x_i^2)^{2^{n-1}}+ \cwlt$.
\item
For each non-identity permutation $\tau$ of $2^{[n]}$, each
monomial in the product $\prod_{I\subset [n]}m_{I \tau(I)}$ is
cwlt $(\prod_{i=1}^n x_i^2)^{2^{n-1}}$.
\end{enumerate}
\end{proposition}

Recall that $2$ is invertible in the commutative domain $k$.

\begin{theorem}\label{xxthm4.9}
Let $B=V_n(\mathcal{A})$ and
$R=k[x_1^2,\dots,x_n^2]\subset B$.
\begin{enumerate}
\item
The discriminant satisfies $d(B/R)=c(\prod_{i=1}^n
x_i^2)^{2^{n-1}}+ \cwlt$ where $c\in k^\times$. As a consequence,
$d(B/R)$ is a dominating element of $B$.
\item
If $g\in \Aut(B)$ is an automorphism so that $g$ and $g^{-1}$
preserve $R$, then $g$ is affine.
\item
If $n$ is even, then $V_n(\mathcal{A})$ is in $\Af$.
\end{enumerate}
\end{theorem}

\begin{proof} (1) By definition, $d(B/R)$ is the determinant of $M$,
which is equal to
\[
\sum_{\tau \in S_{2^n}} (-1)^{|\tau|} \prod_{I\subset [n]}
m_{I\tau(I)}.
\]
In every summand, by Proposition
\ref{xxpro3.7}(2,3), $\prod_{I\subset [n]} m_{II}$ has the highest
possible degree and it is equal to $c(\prod_{i=1}^n
x_i^2)^{2^{n-1}}+ \cwlt$ for some $c\in k^\times$. Any other term
$\prod_{I\subset [n]} m_{I\tau(I)}$, for a non-identity
permutation $\tau$, is a
linear combination of monomials that are cwlt $(\prod_{i=1}^n
x_i^2)^{2^{n-1}}$ by Proposition \ref{xxpro3.7}(3). Therefore
\[
\sum_{\tau \in S_{2^n}} (-1)^{|\tau|} \prod_{I\subset [n]}
m_{I\tau(I)}=c(\prod_{i=1}^n x_i^2)^{2^{n-1}}+ \cwlt
\]
and the
assertion follows.

(2) Assume that $g$ is an automorphism such that $g$ and $g^{-1}$
preserve $R$. By Lemma \ref{xxlem1.8}(f),
$g(d)=c d$ for some $c\in k^\times$. By part (1), $d(B/R)$ is
dominating. By Lemma \ref{xxlem2.6}, $g$ is affine.

(3) This follows from Lemma \ref{xxlem4.1}(5) and part (1).
\end{proof}

When $n$ is odd, part (3) no longer holds. See Example \ref{xxex5.12}
and Remark \ref{xxrem5.14} for more about what happens when $n$ is odd
or when $\ch k = 2$.

Now we are ready to prove Theorem \ref{xxthm0.1}, as well as the
following.

\begin{theorem}\label{xxthm4.10}
Assume that $n$ is a positive even integer. Then $k_{-1}[x_1,\dots,x_n]$
is in $\Af$ and the following hold.
\begin{enumerate}
\item
$\Aut(k_{-1}[x_1,\dots,x_n])= S_n \ltimes (k^\times)^n$.
\item $\Aut(k_{-1}[x_1,\dots,x_n][t])=\begin{pmatrix}
S_n\ltimes (k^\times)^n & k[x_1^{2}, \cdots,x_{n}^2]\\
0& k^\times
\end{pmatrix}$.
\item
If ${\mathbb Q}\subseteq k$, then every locally nilpotent derivation
of $k_{-1}[x_1,\dots,x_n]$ is zero.
\end{enumerate}
\end{theorem}

\begin{proof}[Proof of Theorems \ref{xxthm0.1} and \ref{xxthm4.10}]
Let $B=W_n$ and $R=k[x_1^2,\dots,x_n^2]$. Note that $W_n$ is a special
case of $V_n(\mathcal{A})$. By Theorem \ref{xxthm4.9}(3), $W_n$
is in $\Af$.  Theorem \ref{xxthm0.1} follows from Theorem \ref{xxthm0.3}
and Lemma \ref{xxlem4.3}.

Now consider $B=k_{-1}[x_1,\dots,x_n]$. The first part of the proof is the
same as for $B=W_n$. By Theorem \ref{xxthm4.9}(3), $B$ is in $\Af$ and
every $g\in \Aut(B)$ is affine by Theorem \ref{xxthm0.3}. By
Lemma \ref{xxlem4.3}, there is a $\sigma\in S_n$ such that $g(x_i)=r_i
x_{\sigma(i)}$, where $r_i\in k^\times$. Thus part (1)
follows. Parts (2,3)  follow from Theorem \ref{xxthm0.3}.
\end{proof}

We also have the following results, which follow immediately from
Theorems \ref{xxthm0.1} and \ref{xxthm0.3} and Proposition
\ref{xxpro3.7}.

\begin{theorem}
 \label{xxthm4.11}
 Let $n$ be a positive even integer and $m$ a positive integer.
 \begin{enumerate}
  \item $\Aut(W_2[t])=\begin{pmatrix}
                  S_2\ltimes k^\times & k[x_1^2,x_2^2]\\
                  0& k^\times
                 \end{pmatrix}$.
  \item $\Aut(W_2 [\underline{t}_m^{\pm 1}])
=\left(S_2\ltimes (k[\underline{t}_m])^\times\right)
  \times \left((k^\times)^m \rtimes \{\pm 1\}\right)$.
  \item If $n\geq 4$,
$\Aut(W_n[t])=\begin{pmatrix}
                  S_n\times \{\pm 1\} & k[x_1^2,\dots,x_n^2]\\
                  0& k^\times
                 \end{pmatrix}$.
  \item
  If $n\geq 4$,
$\Aut(W_n [\underline{t}_m^{\pm 1}])=\left(S_n\times \{\pm 1\}\right)
  \times \left((k^\times)^m \rtimes GL_m({\mathbb Z})\right)$.
  \item
  If ${\mathbb Q}\subseteq k$, then every locally nilpotent derivation
of $W_n$ is zero.
 \end{enumerate}
\end{theorem}

Further results can be found in \cite{CPWZ2}.

\begin{question}\label{xxque4.12}
In the above we don't need the exact computation of the
discriminant $d(W_n/R)$, but it would be nice to have.
Let $x_{123\cdots n} = \Omega(\{x_1,\dots,x_n\})$ be defined
as in \eqref{eqn-omega}, let
\[
M= \begin{pmatrix} 2x_1^2& 1&\cdots &1\\
1& 2x_2^2&\cdots&1\\
\vdots &\vdots & \cdots & \vdots\\
1&1& \cdots &2x_n^2\end{pmatrix},
\]
and let $D = \det M$. We have the following questions (or conjectures).
\begin{enumerate}
\item
Is $x_{123\cdots n}^2=_{k^\times} D$?
\item
Is $d(W_n/k[x_1^2,\dots,x_n^2])=_{k^{\times}} D^{2^{n-1}}$?
\end{enumerate}
Both formulas have been verified by computer for even integers $n \leq 6$ (see
also Example~\ref{xxex1.7}(1) for $n=2$). It also appears that if we use
the basis
\[
\{\Omega(x_{i_1},\dots,x_{i_s}) \mid i_{1} < \dots < i_{s} \},
\]
ordered by $s$, to compute the discriminant, then the corresponding
matrix of traces is block diagonal, and the $j$th block is the matrix
of $j \times j$ minors of $M$. Verifying this last statement would
give the above computation of the discriminant, by the
Sylvester-Franke theorem (see \cite{To}, for example).
\end{question}

\section{Comments and examples}
\label{xxsec5}

In this section we provide some comments, remarks, examples and questions
related to automorphisms. To save space, some details are omitted.
By Theorem \ref{xxthm0.3}, if $A$ is in the category $\Af$,
then we can compute its automorphism group. In this section
we would like to show that there are many algebras in $\Af$.

First of all, a dominating discriminant may be in a
form different from the one given in Lemma \ref{xxlem2.2}(1).

\begin{example}
 \label{xxex5.1}
Consider the algebra $S(p):=k\langle x,y \rangle/
(y^2 x-p x y^2,y x^2+p  x^2 y)$ where $p\in k^\times$.
Suppose $k$ is a field.
By \cite[(8.11)]{AS}, $S(p)$
is a noetherian Artin-Schelter regular domain of global
dimension 3, which is of type $S_2$ in the classification
given in \cite{AS}.  Setting $\deg x=\deg y=1$,
$S(p)$ is graded
and its Hilbert series is
\[
H_{S(p)}(t)=\frac{1}{(1-t)^2(1-t^2)}.
\]
It is known that $\GKdim S(p)=\Kdim S(p)=3$.
We are interested in the case when $p=1$, so we set $A=S(1)$.
One can check that the center of $A$ is the commutative
polynomial subring  $R:=k[x^4, y^2,\Omega]$ where
$\Omega= (xy)^2+(yx)^2$.
As an $R$-module, $A$ is free of rank 16. A computation (omitted)
shows that
\[
d(A/R)=_{k^\times} (x^4)^8
(\Omega^2+4x^4 y^4)^8.
\]
We claim that this element is dominating.

Note that, in the algebra $A$, $d(A/R)$ has different presentations
\[
(x^4)^8  (\Omega^2+4x^4 y^4)^8
=(x^4)^8 (xy+i yx)^{32}=(x^4)^8 (xy-i yx)^{32}
\]
where $i^2=-1$. Let $B$ be any ${\mathbb N}$-filtered algebra
such that $\gr B$ is a domain. Let $y_1,y_2$ be any elements
in $B$ of degree at least 1. If $\deg y_1>1$ or $\deg y_2>1$,
then either $\deg(y_1 y_2-iy_2y_1)>2$ or
$\deg(y_1 y_2+iy_2y_1)>2$. Assume the former by symmetry.
Then $\deg (y_1^4)^8 (y_1y_2-i y_2y_1)^{32}> \deg d(A/R)$.
Therefore $d(A/R)$ is dominating. Consequently, $A$ is in $\Af$
and Theorem \ref{xxthm0.3} applies. One can then easily check that
$\Aut(A)=(k^\times)^2$.
\end{example}

Next we show that $\Af$ is closed under tensor products.
We start with a few easy lemmas.

\begin{lemma}
\label{xxlem5.2}
Let $A$ and $B$ be algebras such that their centers $C(A)$ and $C(B)$
are $k$-flat. Then $C(A\otimes B)=C(A)\otimes C(B)$.
\end{lemma}

\begin{lemma}
\label{xxlem5.3} Suppose that $A$ is a free module over $C(A)$ of rank
$m$, and $B$ is a free module over $C(B)$ of rank $n$. Assume that
both $C(A)$ and $C(B)$ are flat over $k$. Then $A\otimes B$ is a free
module over $C(A\otimes B)$ of rank $mn$ and
\[
d(A\otimes B/C(A\otimes B))=d(A/C(A))^n d(B/C(B))^m.
\]
\end{lemma}

\begin{proof} Pick a basis $\{x_i\}$ of $A$ over $C(A)$ and basis
$\{y_j\}$ of $B$ over $C(B)$. For any $a\in A, b\in B$, write $ax_i=\sum_{i'}
r_{ii'} x_{i'}$ and $by_j=\sum_{j'} s_{jj'} y_{j'}$. Then
$\tr(a)=\sum_i r_{ii}$ and $\tr(b)=\sum_j s_{jj}$. Using $\{x_i\otimes y_j\}$
as a basis of $A\otimes B$ over $C(A)\otimes C(B)$, we have
\[
(a\otimes b)(x_i\otimes y_j)=\sum_{i'}\sum_{j'} r_{ii'} s_{jj'}
x_{i'}\otimes y_{j'}
\]
which implies that $\tr(a\otimes b)=\sum_i \sum_j r_{ii}s_{jj}=\tr(a)\tr(b)$.
Now
\begin{align*}
d(A\otimes B/C_A\otimes C_B)&=\det (\tr((x_i\otimes y_j)
(x_{i'}\otimes y_{j'})))
=\det (\tr(x_i x_{i'})\tr(y_j y_{j'}))\\
&=\det (\tr(x_i x_{i'}))^n
\det(\tr(y_jy_{j'}))^m=d(A/C(A))^n d(B/C(B))^m.
\end{align*}
\end{proof}

\begin{lemma}
\label{xxlem5.4} Retain the hypotheses of Lemma \ref{xxlem5.3}.
Suppose that $d(A/C(A))$ and $d(B/C(B))$ are
dominating. Then  so is $d(A\otimes B/C(A\otimes B))$.
\end{lemma}

\begin{proof} Since $d(A/C(A))$ is
a dominating element, $A\neq C(A)$ (unless $A=k$), so $m:=\rk(A/C(A))>1$.
Similarly, $n:=\rk(B/C(B))>1$.
By Lemma \ref{xxlem5.3}, the discriminant of $A\otimes B$
over its center is $d(A/C(A))^n d(B/C(B))^m$. By hypothesis,
both $d(A/C(A))$ and $d(B/C(B))$ are dominating, and it is routine to
check that $d(A/C(A))^n d(B/C(B))^m$ is dominating.
\end{proof}

\begin{theorem}
 \label{xxthm5.5} Retain the hypotheses of Lemma \ref{xxlem5.3}.
Suppose that $\gr A\otimes \gr B$ is a connected graded domain.
If $A$ and $B$ are in $\Af$, so is $A\otimes B$.
\end{theorem}

\begin{proof}
This follows from Lemmas \ref{xxlem5.3} and \ref{xxlem5.4}.
\end{proof}

Another property of $\Af$ is that, if $A$ is in $\Af$, then so is
the opposite ring of $A$, denoted by $A^{\op}$. We identity $A^{\op}$
with $A$ as a $k$-module, and the multiplication of $A^{\op}$, denoted
by $*$, is defined by
\[
a*b= ba \quad \forall \; a,b\in A.
\]
The regular trace of $A^{\op}$ is denoted by $\tr^{\op}_{\reg}$.
We may also use right multiplication on $A$ to define a right-hand
version of the regular trace, denoted by $\tr'_{\reg}$.

\begin{lemma}
\label{xxlem5.6}
Let $A$ be a PI domain and $F$ be the field of fractions of the center
$C(A)$. Let $\tr: A\to F$ be a trace function.
\begin{enumerate}
 \item
$\tr$ is uniquely determined by $\tr(1)$.
\item
Identifying $A^{\op}$ with $A$ as a $C(A)$-module, then
$\tr_{\reg}=\tr'_{\reg}=\tr^{\op}_{\reg}$.
\end{enumerate}
\end{lemma}

\begin{proof} (1) Let $R=C(A)$. Then the $R$-linear trace $\tr: A\to F$
can be extended to an $F$-linear trace $\tr: A\otimes_R F\to F$
uniquely. So we may assume that $A$ is a division ring
with center $F$. It is well-known that the trace on a simple
algebra is uniquely determined by $\tr(1)$ (by using a spliting
field). Therefore $\tr$ is uniquely determined by $\tr(1)$ by
restriction.

(2) Since $\tr'_{\reg}(1)=\rk(A/C(A))$, the first equality follows
from part (1). The second equality follows from the definition.
\end{proof}

\begin{proposition}
 \label{xxpro5.7} Let $A$ be in $\Af$ and let $B=A^{\op}$
 \begin{enumerate}
  \item
$d(A/C(A))=d(B/C(B))$.
\item
$B$ is in $\Af$.
\item
Every anti-automorphism $g$ of $A$ is affine, namely, $g(F_1 A)\subset F_1 A$.
\item
Every anti-automorphism $h$ of the polynomial extension $A[t]$
is triangular, namely, there is an anti-automorphism $g$ of $A$,
$c\in k^\times$ and
$r\in R$ such that
\[
h(t)=ct+r \quad {\text{and}}\quad
h(x)=g(x)\in A \quad {\text{for all $x\in A$}}.
\]
 \end{enumerate}
\end{proposition}

\begin{proof}
(1) This follows from Lemma \ref{xxlem5.6}(2) and the definition.

(2) Follows from part (1) and the definition.

(3) Modifying the proof of Lemma \ref{xxlem1.8}, one sees that
$g(\tr_{\reg}(x))=\tr'_{\reg} (g(x))=\tr_{\reg}(g(x))$ where
 the second equality is Lemma \ref{xxlem5.6}(2). By the proof of
 Lemma \ref{xxlem1.8}(3), one sees that $d(A/C(A))$ is
 $g$-invariant up to a scalar in $k^\times$. Modifying
 the proof of Lemma \ref{xxlem2.6} and using the dominating
 element $d(A/C(A))$ and the testing algebra $T=A^{\op}$, one
 can show that $g$ is affine.

(4) Modify the original proof for automorphisms and use an
idea similar to the proof of part (3). Details are omitted.
\end{proof}

Therefore the following algebras are in $\Af$:
\begin{enumerate}
 \item
All $V_n({\mathcal A})$ when $n$ is even [Theorem \ref{xxthm4.9}].
Special cases are $k_{-1}[x_1,\dots,x_n]$ and $W_n$ when $n$ is even.
\item
$A=k\langle x,y \rangle/
(y^2 x-x y^2,y x^2+x^2 y)$ [Example \ref{xxex5.1}].
\item
Any skew polynomial ring $A=k_{p_{ij}}[x_1,\dots,x_n]$ satisfying the
properties that (a) $x_i$ are not central for all $i$ and (b) $A$ is a
finitely generated free module over its center \cite{CPWZ2}.
\item
Quantum Weyl algebras $A_q:=k\langle x,y\rangle/(yx-qxy-1)$
where $q\neq 1$ and $q$ is a root of unity \cite{CPWZ2}.
\item
Any tensor product of the algebras listed above.
\item
Any opposite ring of $A$ in $\Af$ is again in $\Af$.
\end{enumerate}

In Section~\ref{xxsec2} we used standard filtrations in the
definitions of dominating elements and affine automorphisms. In
practice one might have to use non-standard filtrations in order to
determine automorphism groups. Here is an example.

\begin{example}
\label{xxex5.8}
Suppose $2$ is invertible in $k$.
Let $D$ be the fixed subring $k_{-1}[x_1,x_2]^{S_2}$
where the group $S_2$ is generated by the permutation
$\sigma: x_1\leftrightarrow x_2$. Hence $D$ is a graded PI
domain. A presentation of $D$ is given by
\[
D\cong k\langle x,y\rangle/(x^2 y- yx^2, xy^2-y^2 x,
 2x^6-3 x^3 y - 3 y x^3 + 4 y^2)
\]
where $x=x_1+x_2, y=x_1^3+x_2^3$ \cite[Example 3.1]{KKZ}.
Replacing $y$ by $4y-3x^3$, $D$ has a better
presentation
\[
D\cong k\langle x,y\rangle/(x^2 y- yx^2, xy^2-y^2 x,
 x^6-y^2)
\]
which we will use for the rest of this example.  Then $D$ is a
connected graded algebra with $\deg x=1$ and $\deg y=3$. If we use a
standard filtration for any possible generating set $Y$, the
associated graded ring will not be a domain due to the third
relation. Therefore it is not a good idea to use the standard
filtration as we need to use \eqref{2.0.1} in our argument. A
computation shows that the center of $D$ is the polynomial ring
generated by $x^2$ and $z:=xy+yx$, and the discriminant of $d(D/C(D))$
is $f:=(xy-yx)^4$. Using the relations of $D$, one has
\[
f=((xy-yx)^2)^2=((xy+yx)^2-4x^2y^2)^2=(z^2-4x^8)^2=
(z-2x^4)^2(z+2x^4)^2.
\]
Let $g$ be any automorphism of $D$. By Lemma \ref{xxlem1.8}(6),
$g(f)=cf$ for some $c\in k^\times$. Since the polynomial ring
$k[x^2,z]$ is a unique factorization domain, we have
\[
\begin{cases}
g(z-2x^4) = a (z-2 x^4)&\\
g(z+2x^4) = b (z+2 x^4)&
\end{cases} \quad {\text{or}}\quad
\begin{cases}
g(z-2x^4) = a (z+2 x^4)&\\
g(z+2x^4) = b (z-2 x^4)&
\end{cases}
\]
for some $a,b\in k^\times$. Hence $g(z\pm 2x^4)$ has degree 4.
Consequently, $g(x^4)$ has degree (at most) 4,
which implies that $g(x)$ has degree 1.
By the third relation of $D$,
$g(y)$  has degree 3. From this it is easy to check that
\[
\begin{cases}
g(x) = a x&\\
g(y) = a^3 y&
\end{cases} \quad {\text{or}}\quad
\begin{cases}
g(x) = a x&\\
g(y) =- a^3 y&
\end{cases}
\]
for some $a\in k^\times$. Therefore
$\Aut(D)=k^\times \rtimes S_2$.

We could modify the definition of $\Af$ so that $D$ is
in the category $\Af$, but the definition would be more
complicated in order to keep the tensor product
property [Theorem \ref{xxthm5.5}]. At this point we
would like to treat $D$ separately.  We have checked that
all conclusions of Theorem \ref{xxthm0.3} hold for $D$.

Note that $k_{-1}[x_1,x_2]$ is in $\Af$ and
$D=k_{-1}[x_1,x_2]^{S_2}$. We may ask the following question:
does $k_{-1}[x_1,\dots,x_{2m}]^{S_{2m}}$ have an
``affine'' automorphism group for all $m\geq 2$?
\end{example}

\begin{example}
\label{xxex5.9}
Let $\ell\geq 3$ and $q$ be a primitive $\ell$th root of unity.
Let $A$ be the algebra $(k_q[x_1,x_2])[x_3]$.  Then $A$ is a connected graded
domain with $\deg (x_i)=1$ for $i=1,2,3$. Since $x_3$ is central,
it is not hard to check that the center of $A$ is
$R=k[x_1^{\ell},x_2^{\ell},x_3]$. Hence $A$ is finitely generated
free over its center with an $R$-basis $\{x_1^a x_2^b \mid 0\leq a,b
\leq \ell-1\}$. Therefore (1) and (2) of Definition \ref{xxdef2.4} hold.
By a computation, the discriminant $d(A/R)$ is equal to
$(x_1 x_2)^{\ell^2 (\ell-1)}$, which is not dominating. Therefore
(3) of Definition \ref{xxdef2.4} fails. With some
effort, one can show that every automorphism $g$ of $A$ is of the form
\[
g(x_i)=\begin{cases} a_1 x_1 & i=1, \\
a_2 x_2 &i=2, \\ a_3 x_3+f(x_1^\ell,x_2^\ell)& i=3,\end{cases}
\]
where $a_i\in k^\times$ and $f$ is a polynomial of two variables, and every
locally nilpotent derivation $\partial$ of $A$ is of the form
\[
\partial(x_i)=\begin{cases}\quad 0\quad &\quad i=1, \\
\quad0 \quad & \quad i=2, \\\quad
f(x_1^\ell,x_2^\ell) \quad &\quad i=3.\end{cases}
\]

\end{example}

By Theorem \ref{xxthm0.3}(4), if $k$ is a field, then
$\Aut: A\mapsto \Aut(A)$ defines a functor from $\Af$ to the
category of algebraic groups over $k$. There are some interesting questions
about this functor. It is well-known that the symmetry index $si$
(defined after Theorem~\ref{xxthm0.3})
is neither additive nor multiplicative. For example,
if $A$ and $A^{\otimes n}$ are both in $\Af$, then
$si(A^{\otimes n})\geq n! (si(A))^n$. What about the
symmetry rank?

\begin{question}
 \label{xxque5.10} Let $k$ be a field and let $A$ and $B$ be in $\Af$.
 Is $sr(A\otimes B)
 =sr(A)+sr(B)$?
\end{question}

\begin{remark}
 \label{xxrem5.11} In \cite{CPWZ2} we use the discriminant to propose another
 category $\Af_{-1}$ that has the following properties:
 \begin{enumerate}
  \item
  If $A$ is in $\Af$, then the polynomial extension $A[t]$ is
  in $\Af_{-1}$ (and there are many other algebras in $\Af_{-1}$),
  \item
  If $B$ is in $\Af_{-1}$, then $\Aut(B)$ is tame.
 \end{enumerate}
Therefore the automorphism groups of the algebras in $\Af_{-1}$
can be understood (in theory).
\end{remark}

We now consider $W_n=k\langle x_1, \dots, x_n \rangle /
(x_ix_j+x_jx_i-1, \, \forall \, i\neq j)$ again, when $n$ is odd or
$\ch k=2$.

\begin{example}\label{xxex5.12}
Consider the standard filtration of $W_n$  defined
by $Y=\bigoplus_{i=1}^n kx_i$. As stated in Theorem \ref{xxthm0.1},
if $n$ is even and $\ch k\neq 2$, then every automorphism of
$W_n$ is affine.  Here are some examples of non-affine automorphisms
in other cases.
\begin{enumerate}
\item
If $\ch k=2$, then for any nonzero polynomial
$f(t_1,\dots,t_{n-1})$, the following determines a non-affine
algebra automorphism of $W_n$:
\[
x_i \mapsto
\begin{cases}
x_i & \text{if $i<n$}, \\
x_n+f(x_1^2,\dots,x_{n-1}^2) & \text{if $i=n$}.
\end{cases}
\]
The associated locally nilpotent derivation is determined by
\[
x_i \mapsto
\begin{cases}
0& \qquad \text{if $i<n$}, \\
f(x_1^2,\dots,x_{n-1}^2) &\qquad  \text{if $i=n$}.
\end{cases}
\]
\item
As in ~\eqref{eqn-omega}, define
\[
\Omega(x_1,x_2,\dots, x_n)=\sum_{\sigma\in S_n}
(-1)^{|\sigma|} x_{\sigma(1)}\cdots x_{\sigma(n)}.
\]
Then we claim that $x_i
\Omega(x_1,\dots,x_{2m})=-\Omega(x_1,\dots,x_{2m})x_i$
for all $i=1,2,\dots, 2m$: see Lemma \ref{xxlem5.13} below.
Given this, if $n$ is odd, say $n=2m+1$, then
for any nonzero polynomial $f(t_1,\dots,t_{2m})$,
the following determines
a non-affine algebra automorphism $\sigma$ of $W_n$:
\[
x_i \mapsto
\begin{cases}
x_i & \text{if $i<2m+1$}, \\
x_{2m+1}+f(x_1^2,\dots,x_{2m}^2)
\Omega(x_1,\dots,x_{2m}) & \text{if $i=2m+1$.}
\end{cases}
\]
The associated locally nilpotent derivation $\partial$ is determined by
\[
x_i \mapsto
\begin{cases}
\quad 0 \quad & \qquad \text{if $i<2m+1$}, \\
\quad f(x_1^2,\dots,x_{2m}^2)
\Omega(x_1,\dots,x_{2m}) \quad & \qquad \text{if $i=2m+1$},
\end{cases}
\]
and $\sigma=\exp(\partial)$.
\end{enumerate}
The automorphisms in (1) and (2) are examples of
\emph{elementary} automorphisms -- see \cite{SU}.
\end{example}

\begin{lemma}
\label{xxlem5.13} Let $W_n$ and $\Omega_{n} := \Omega(x_1,\dots,x_n)$
be defined as in Example \textup{\ref{xxex5.12}}.
Then $x_i \Omega_n=(-1)^{n-1}\Omega_n x_i$
for all $i=1,2,\dots, n$.
\end{lemma}

\begin{proof}
It is easy to reduce to the case when
$k={\mathbb Z}$.

We proceed by induction. It is easy to check that the
assertion holds when $n=2$ by using the fact that
$x_i^2$ is central. Now assume the assertion holds
for $n-1\geq 2$ and we want to show that it holds for $n$.
Note that, for every $\sigma\in S_n$,
$\Omega(x_{\sigma(1)},\dots, x_{\sigma(n)})=(-1)^{|\sigma|}
\Omega(x_1,\dots,x_n)$. By symmetry, it suffices to show
that $x_1 \Omega_n =(-1)^{n-1} \Omega_n x_1$.
The argument below is dependent on the parity of $n$, and we only give
a proof when $n$ is odd. The proof when $n$ is even is very
similar, and we omit it. Since $n$ is odd, it suffices to show that
$x_1\Omega_n-\Omega_n x_1=0$. We compute
 $x_1\Omega_n-\Omega_n x_1$ in two different ways.

It follows from the definition that
$\Omega_n=\sum_{i=1}^n (-1)^{i-1} x_i
\Omega(x_1,\dots,\widehat{x}_i,\dots,x_n)$.
Then, by using the induction hypothesis,
\[
\begin{aligned}
x_1\Omega_n&-\Omega_n x_1\\
&=x_1(x_1
\Omega(\widehat{x}_1,x_2,,\dots,x_n))-x_1\Omega(\widehat{x}_1,x_2,,\dots,x_n)x_1\\
&\quad +\sum_{i\geq 2} (-1)^{i-1} [x_1 x_i
\Omega(x_1,\dots,\widehat{x}_i,\dots,x_n)-x_i
\Omega(x_1,\dots,\widehat{x}_i,\dots,x_n)x_1]\\
&=x_1^2
\Omega(\widehat{x}_1,x_2,,\dots,x_n)-x_1\Omega(\widehat{x}_1,x_2,,\dots,x_n)x_1\\
&\quad +\sum_{i\geq 2} (-1)^{i-1} [x_1 x_i
\Omega(x_1,\dots,\widehat{x}_i,\dots,x_n)+x_ix_1
\Omega(x_1,\dots,\widehat{x}_i,\dots,x_n)]\\
&=x_1^2
\Omega(\widehat{x}_1,x_2,,\dots,x_n)-x_1\Omega(\widehat{x}_1,x_2,,\dots,x_n)x_1\\
&\quad +\sum_{i\geq 2} (-1)^{i-1}
\Omega(x_1,\dots,\widehat{x}_i,\dots,x_n).
\end{aligned}
\]
On the other hand,
\[
\begin{aligned}
\Omega_n&=\sum_{i=1}^n (-1)^{i-n} \Omega(x_1,\dots,\widehat{x}_i,\dots,x_n) x_i\\
&=\sum_{i=1}^n (-1)^{i-1} \Omega(x_1,\dots,\widehat{x}_i,\dots,x_n) x_i
\end{aligned}
\]
as $n$ is odd. So we have
\[
\begin{aligned}
x_1\Omega_n&-\Omega_n x_1\\
&=x_1(\Omega(\widehat{x}_1,x_2,,\dots,x_n)x_1)-\Omega(\widehat{x}_1,x_2,,\dots,x_n)x_1^2\\
&\quad +\sum_{i\geq 2} (-1)^{i-1} [x_1
\Omega(x_1,\dots,\widehat{x}_i,\dots,x_n)x_i-
\Omega(x_1,\dots,\widehat{x}_i,\dots,x_n)x_ix_1]\\
&=-x_1^2
\Omega(\widehat{x}_1,x_2,,\dots,x_n)+x_1\Omega(\widehat{x}_1,x_2,,\dots,x_n)x_1\\
&\quad +\sum_{i\geq 2} (-1)^{i-1} [-
\Omega(x_1,\dots,\widehat{x}_i,\dots,x_n)x_1 x_i-
\Omega(x_1,\dots,\widehat{x}_i,\dots,x_n)x_ix_1]\\
&=-x_1^2
\Omega(\widehat{x}_1,x_2,,\dots,x_n)+x_1\Omega(\widehat{x}_1,x_2,,\dots,x_n)x_1\\
&\quad +\sum_{i\geq 2} (-1)^{i}
\Omega(x_1,\dots,\widehat{x}_i,\dots,x_n)\\
&=-(x_1\Omega_n-\Omega_n x_1).
\end{aligned}
\]
Since $2\neq 0$ in ${\mathbb Z}$, $x_1\Omega_n-\Omega_n x_1=0$ as required.
\end{proof}

\begin{remark}
\label{xxrem5.14}
By the previous example, when $n$ is odd or when $\ch k = 2$,
there are non-affine automorphisms. Thus the automorphism group looks
complicated. Also, it appears that the discriminant does not provide
useful information in either case: a (nontrivial) computation shows
that the discriminant ideal of $W_3$ over its center contains 1, and
hence it is trivial.  We
conjecture that this holds for any odd integer $n\geq 3$.  We also
note when $n$ is odd, the center $R$ contains $\Omega(x_{1}, \dots, x_{n})$, so
$W_{n}$ is not free over $R$.  When $\ch k = 2$, Lemma~\ref{xxlem4.5}(1)
says that $\tr (1)=0$ in $k$, and computer calculations suggest that
the discriminant is zero (whence trivial)  in general.  (For more evidence, see
Remark~\ref{xxrem4.7} -- some of these computations remain valid in
characteristic 2.) In conclusion, new invariants are needed to understand
(or control) $\Aut(W_n)$ when $n$ is odd or when $\ch k = 2$.
\end{remark}

We conclude this paper with the following question.

\begin{question}
\label{xxque5.15} If $n$ is odd and/or $\ch k=2$, what is the group
$\Aut(W_n)$?
\end{question}

\subsection*{Acknowledgments}
The authors would like to thank   Ken Goodearl, Colin Ingalls,
Rajesh Kulkarni, and Milen Yakimov for several conversations on this
topic during the Banff workshop in October 2012 and the NAGRT
program at MSRI in the Spring of 2013. S. Ceken was supported by the
Scientific and Technological Research Council of Turkey (TUBITAK),
Science Fellowships and Grant Programmes Department (Programme no.
2214). Y.H. Wang was supported by the Natural Science Foundation of
China (grant no. 10901098, 11271239). J. J. Zhang was supported by
the US National Science Foundation (NSF grant No. DMS 0855743).

\providecommand{\bysame}{\leavevmode\hbox to3em{\hrulefill}\thinspace}
\providecommand{\MR}{\relax\ifhmode\unskip\space\fi MR }
\providecommand{\MRhref}[2]{%

\href{http://www.ams.org/mathscinet-getitem?mr=#1}{#2} }
\providecommand{\href}[2]{#2}

\end{document}